\documentclass[a4paper,10pt,reqno]{amsart}
\usepackage{amsfonts,amssymb,amscd,amsmath,latexsym,amsbsy,enumerate,a4wide,
	tikz,  ifthen, mathrsfs, mathtools}

\usepackage[all]{xy}
\usetikzlibrary{calc}
\usepackage{tikz-cd}

\newtheorem{thm}{Theorem}[section]
\newtheorem{defn}[thm]{Definition}
\newtheorem{lem}[thm]{Lemma}
\newtheorem{cor}[thm]{Corollary}
\newtheorem{prop}[thm]{Proposition}
\newtheorem{example}[thm]{Example}
\theoremstyle{definition}
\newtheorem{rmk}[thm]{Remark}

\numberwithin{equation}{section}

\def\al{\alpha}

\def\C{\mathbb{C}}

\newcommand{\D}{\mathcal{D}}

\begin{document}
	\title{Lie algebras of differential operators for Matrix valued Laguerre type polynomials}
	\author[A.L.\@ Gallo, P.\@ Rom\'an]{Andrea L.\@ Gallo, Pablo Rom\'an}
	\dedicatory{\today}
	\thanks{Partially supported by CONICET, FONCyT and SECyT-UNC}
	
	\begin{abstract}
		We study algebras of differential and difference operators acting on matrix valued orthogonal polynomials (MVOPs) with respect to a weight matrix of the form  $W^{(\nu)}_{\phi}(x) = x^{\nu}e^{-\phi(x)} W^{(\nu)}_{pol}(x)$, where $\nu>0$, $W^{(\nu)}_{pol}(x)$ is certain matrix valued polynomial and $\phi$ an entire function. We introduce a pair differential operators $\mathcal{D}$, $\mathcal{D}^{\dagger}$ which are mutually adjoint with respect to the matrix inner product induced by $W^{(\nu)}_{\phi}(x)$. We prove that the Lie algebra generated by $\mathcal{D}$ and $\mathcal{D}^{\dagger}$ is finite dimensional if and only if $\phi$ is a polynomial, giving a partial answer to a problem by M. Ismail. In the case $\phi$ polynomial, we describe the structure of this Lie algebra.  The case $\phi(x)=x$, is discussed in detail. We derive difference and differential relations for the MVOPs. We give explicit expressions for the entries of the MVOPs in terms of classical Laguerre and Dual Hahn polynomials.
	\end{abstract}
	
	\maketitle
	
	\section{Introduction}
	The theory of matrix valued orthogonal polynomials (MVOPs) was initiated by Krein 1940s, and it has since been used in various areas of mathematics and mathematical physics. These areas include spectral theory, scattering theory, tiling problems, integrable systems, and stochastic processes. For further details and insights on these subjects, refer to \cite{Arniz-2014}, \cite{DamanikPS}, \cite{DuitsK}, \cite{Geronimo}, \cite{GroeneveltIK}, \cite{IglesiaG}, and the references therein.
	
	Significant progress has been made in the past two decades towards understanding how the differential and algebraic properties of classical scalar orthogonal polynomials can be extended to the matrix valued setting. A fundamental role has been played by the connection between harmonic analysis of matrix valued functions on compact symmetric pairs and matrix valued orthogonal polynomials. In \cite {Duran1}, A. Durán poses the problem of determining families of MVOPs which are eigenfunctions of a suitable second order differential operator. In the scalar case, the answer to this problem is a classical result due to Bochner \cite{Boch}. The only families with this property are those of Hermite, Laguerre and Jacobi. The matrix valued setting turns out to be much more involved. The first explicit examples appeared in connection with spherical functions of the compact symmetric pair $(\mathrm{SU}(3),\mathrm{U}(2))$. Following \cite{Koornwinder}, a direct approach was taken in \cite{KvPR1}, \cite{KvPR2} for the case of $(\mathrm{SU}(2) \times \mathrm{SU}(2), \mathrm{diag})$, leading to a general set-up in the context of multiplicity free pairs \cite{HeckmanP}. In this context, certain properties of the orthogonal polynomials, such as orthogonality, recurrence relations, and differential equations, are understood in terms of the representation theory of the corresponding symmetric spaces. Recently, Casper and Yakimov developed a framework in \cite{Casper2} to solve the matrix Bochner problem. This involves the classification of all $N \times N$ weight matrix $W(x)$ whose associated MVOPs are eigenfunctions of a second-order differential operator.

	Given $N \in \mathbb{N}$ we consider a matrix valued function $W: [a,b] \rightarrow M_N(\mathbb{C})$ such that $W(x)$ is positive definite for all $x\in [a,b]$ and $W$ has finite moments of all order. In such a case, we say that $W$ is a weight function,  which induces matrix valued inner product
	\begin{equation}\label{eq:HermitianForm}
		\langle P,Q \rangle=\int_a^b P(x)W(x)Q(x)^*dx \in M_N(\C),
	\end{equation}
	such that for all $P,Q,R\in M_N(\C)[x]$, $T\in M_N(C)$ and $a,b\in \mathbb{C}$ 
	the following properties are satisfied
	$$
	\langle aP+bQ,R\rangle=a\langle P,R\rangle+b\langle Q,R\rangle, \qquad \langle TP,Q\rangle=T\langle P,Q\rangle, \langle P,Q\rangle^*=\langle Q,P\rangle.$$
	Moreover $\langle P,P \rangle = 0$ if and only if $P=0$. Using standard arguments, it can be shown that there exists a unique sequence $(P(x,n))_n$ monic MVOPs with respect to $W$ in the following sense:
	\begin{equation}\label{equation Hn}
		\langle P(x,n), P(x,m)\rangle = \mathcal{H}(n) \delta_{n,m}.    
	\end{equation}
	where the squared norm $\mathcal{H}(n)$ is a positive definite matrix.

	By orthogonality, the polynomials $P(x,n)$'s satisfy the following three-term recurrence
	relation:
	\begin{equation}\label{eq:three_term_monic}
		xP(x,n) = P(x,n+1) + B(n) P(x,n) + C(n) P(x,n-1)
	\end{equation}
	where $B(n), C(n) \in M_N(\mathbb{C})$ and $n \geq 1$.
	Notice that $B(n)$ and $C(n)$ satisfy 
	\begin{equation}\label{B C}
		B(n)=X(n)-X(n+1), \quad C(n)= \mathcal{H}(n) \mathcal{H}(n-1)^{-1},    
	\end{equation}
	where  $X(n)$ is the one-but-leading coefficient of $P(x,n)$ and $\mathcal{H}(n)$ as in \eqref{equation Hn}.
	Moreover, for $n \geq 2$, let $Y(n)$ denotes the second-but-leading coefficient of $P(x,n)$. Then
	\begin{equation}\label{prop Y}
		Y(n)=Y(n+1)+B(n)X(n)+C(n).    
	\end{equation}
	
	In \cite{DERom}, the authors studied difference--differential relations for a specific class of MVOPs associated with the weight $W(x)=e^{-v(x)}e^{xA}e^{xA^{\ast}}$, where $x \in \mathbb{R}$, $v(x)$ is a scalar polynomial of even degree, and $A$ is a constant matrix. There is a way of obtaining information about the matrix orthogonal polynomials by investigating two mutually adjoint operators $\mathcal{D}$ and $\mathcal{D}^\dagger$. If $v(x)$ is a polynomial of degree two, in addition to $\mathcal{D}$ and $\mathcal{D}^\dagger$, there exists a second order differential operator $D$ having the MVOPs as eigenfunctions. It turns out that $\mathcal{D}, \mathcal{D}^\dagger$ and $D$ generate a finite dimensional Lie algebra which is isomorphic to the Lie algebra of the oscillator group. The Casimir operator for this algebra is given explicitly and used to obtain information of the MVOPs.
	In this work, we solve the analogous problem for Laguerre-type weights. This case is more involved than the previous one due to the structure of the associated Lie algebra and the non-diagonality of certain formulas that involve $W$.
	
	In the scalar case, this problem is closely related to \cite[Problem 24.5.2]{Ismail}. Here Ismail proposed to study the finite dimensionality of certain Lie algebra generated by a pair of differential operators which are mutually adjoint respect to a Laguerre-type weight. More precisely, given the scalar weight $w_1(x)=x^{\alpha}e^{-\phi(x)}$ with $x>0, \,\alpha>1$ and differential operators
	$$\mathcal{D}_{1,n}=x\partial_x +x B_{n}(x), \qquad \mathcal{D}_{2,n}=-x\partial_x + x B_{n}(x)+x\phi'(x),$$
	where $\{B_{n}\}$ is a sequence of scalar polynomials, the problem asks to prove that ``\emph{The Lie algebra generated by $\mathcal{D}_{1,n}$ and $\mathcal{D}_{2,n}$ is finite dimensional if and 	only if $\phi$ is a polynomial}''.
	
	In this paper we provide a partial answer to this problem in the context of matrix valued orthogonal polynomials. We give an explicit matrix valued weight $W^{(\nu)}_{\phi}(x) =x^\nu e^{-\phi(x)} W^{(\nu)}_{pol}(x)$, where $W^{(\nu)}_{pol}(x)$ is a matrix polynomial depending on $\nu$, and differential operators 
	$$\mathcal{D} = \partial_x x + x(A-1), \quad 
	\mathcal{D}^\dagger = -\partial_x x - (1+\nu+J)+x\phi'(x)-x.$$ 
	In this case, we prove that the Lie algebra generated by is finite dimensional if and only if $\phi$ is a polynomial.	As a consequence, this solves \cite[Problem 24.5.2]{Ismail} when $B_{n}(x)=-1$ for all $n\in \mathbb{N}$.
	
	\subsection*{Outline and main results}
	
	In Section $2$ we recall some preliminaries. In particular, we introduce the left and right Fourier algebras related to the sequence of monic MVOPs.
	
	In Section $3$ for a given analytic function $\phi$ on a neighborhood of the interval $[0,\infty)$, we introduce a Laguerre type weight $W^{(\nu)}_{\phi}$ and the operators $\mathcal{D}$ , $\mathcal{D}^{\dagger}$ and prove that they are mutually adjoint with respect to $W^{(\nu)}_{\phi}(x)$. For the MVOPs $\{P_n\}$ respect to $W^{(\nu)}_{\phi}$,
	we find discrete operators $M,\, M^{\dagger}$ associated to $\mathcal{D},\, \mathcal{D}^{\dagger}$ respectively, given by the relations $M \cdot P_n = P_n \cdot \mathcal{D}$ and $M^{\dagger} \cdot P_n = P_n \cdot \mathcal{D}^{\dagger}$.
	
	In Section $4$ we study the Lie algebra $\mathfrak{g}_{\phi}$ generated by the differential operators $\mathcal{D}$, $\mathcal{D}^{\dagger}$. We prove that $\mathfrak{g}_{\phi}$ is finite dimensional if and only if $\phi$ is a polynomial. Also, for this family of Lie algebras $\{\mathfrak{g}_{\phi}\}$ we obtain that $\mathfrak{g}_{\phi}=\mathbb{C}^{2} \oplus \mathfrak{h}$ and $\mathfrak{h}$ is a solvable Lie algebra with nilradical of codimension one. Moreover, we obtain a classification of this family of Lie algebra up to isomorphisms.
	
	In Section $5$ we give an explicitly expression for $\mathcal{D}$, $\mathcal{D}^{\dagger}$, $M$ and $M^{\dagger}$ in the case $\phi(x)=x$.	In this case, we also find a symmetric second-order differential operator $D$ which have $\{P_n\}$ as eigenfunctions. We describe the Lie algebra $\mathcal{A}$ generated by $\mathcal{D}$, $\mathcal{D}^{\dagger}$ and $D$, $\mathcal{A}=\mathcal{Z}_{\mathcal{A}} \oplus [\mathcal{A},\mathcal{A}]$ where $\dim \mathcal{Z}_{\mathcal{A}}=2$ and $[\mathcal{A},\mathcal{A}]$ is isomorphic to $\mathrm{SL}(2,\mathbb{C})$. Also, we obtain some relations between $\mathcal{H}_n$, $B_n$ and $C_n$. 
	
	In Section 6, we consider the polynomials $R(x,n)=K_n P(x,n)e^{xA}$ 
	where $P(x,n)$ are the MVOPs associated with the weight $W^{(\nu)}$ and
	$K_n$ certain lower triangular matrices. Using the operator $D$,
	we show that the matrix entries of $R_n$ can be put in terms of 
	generalized Laguerre polynomials  and a family of constants $\xi(n,i,j)$'s. 
	Finally, we give two-terms recursions for the constants $\xi(n,i,j)$'s 
	and for the squared norm $\mathcal{H}_n$.
	
	Finally, in Section $7$, in the case $A=- \sum_{k=1}^{N-1} E_{k+1,k}$, and $\delta^{k}>0$ 
	satisfying two non-linear conditions (related to Pearson's equations),  we show that the constants $\xi(n,i,j)$'s are written in terms of dual Hanh polynomials.
	
	\section{Preliminaries}
	
	\label{sec:pre}
	
	This section presents the left and right Fourier algebras associated with the sequence of monic MVOPs, as developed by Casper and Yakimov in \cite{Casper2}. The results discussed in this section have been previously covered in a more comprehensive context in \cite{Casper2}.
	
	Let $Q(x,n)$ be a function $Q:\mathbb{C}\times \mathbb{N}_0 \to M_N(\mathbb{C})$ such that $Q(x,n)$ is a rational function of $x$ for fixed $n$. A differential operator of the form
	\begin{equation}
		\label{eq:DifferentialOperator}
		\D=\sum_{j=0}^n \partial_x^j F_j(x), \qquad \partial_x^j := \tfrac{d^j}{dx^j},
	\end{equation}
	where $F_j:\mathbb{C}\to M_N(\mathbb{C})$ is a rational function of $x$, acts on $Q$ from the right by
	$$(Q\cdot \D)(x,n)  = \sum_{j=0}^n (\partial_x^jQ)(x,n)\,  F_j(x).
	$$
	
	The algebra of all differential operators of the form \eqref{eq:DifferentialOperator} will be denoted by $\mathcal{M}_N$.  In addition to the right action by differential operators, we also consider a left action on $Q$ by difference operators on the variable $n$. For $j\in\mathbb{Z}$, let $\delta^{j}$ be the discrete operator which acts on a sequence $A:\mathbb{N}_0 \to M_N(\mathbb{C})$ by
	$$(\delta^j \cdot A)(n)=A(n+j).$$
	Here we assume that the value of a sequence at a negative integer is equal to zero. For given sequences $A_{-\ell},\ldots,A_k$, a discrete operator of the form
	\begin{equation}
		\label{eq:DifferenceOperator}
		M=\sum_{j=-\ell}^k A_j(n) \delta^j,
	\end{equation}
	acts on $Q$ from the left by
	\begin{align*}
		(M \cdot Q)(x,n) &= \sum_{j=-\ell}^k A_j(n) \, (\delta^j\cdot Q)(x,n) = \sum_{j=-\ell}^k A_j(n) \, Q(x,n+j).
	\end{align*}
	
	We shall denote the algebra of difference operators \eqref{eq:DifferenceOperator} by $\mathcal{N}_N$. As in \cite[Definition 2.20]{Casper2} we define:
	\begin{defn}
		The left and right Fourier algebras are given by:
		\begin{equation}
			\label{eq:definition-Fourier-algebras}
			\begin{split}
				\mathcal{F}_L(P)&=\{ M\in \mathcal{N}_N \colon \exists\, \D\in \mathcal{M}_N,\, M\cdot P = P\cdot \D \} \subset \mathcal{N}_{N},\\
				\mathcal{F}_R(P)&=\{ \D\in \mathcal{M}_N \colon \exists\, M\in \mathcal{N}_N,\, M\cdot P = P\cdot \D \}\subset \mathcal{M}_{N}.
			\end{split}
		\end{equation}
	\end{defn}
	
	The definition of the Fourier algebras directly implies a connection between the elements of $\mathcal{F}_L(P)$ and $\mathcal{F}_R(P)$. Moreover, the map
	\begin{equation*}
		\varphi\colon \mathcal{F}_L(P) \to \mathcal{F}_R(P),\qquad \text{ defined by }\quad M\cdot P = P \cdot \varphi(M),
	\end{equation*}
	is an algebra isomorphism. In \cite{Casper2} this map is called the \textit{generalized Fourier map}. More precisely, $M_{1}M_{2}\cdot P = P \cdot \varphi(M_{1})\varphi(M_{2})$ for all $M_{1},M_{2}\in \mathcal{F}_{L}(P)$. On the other hand, by the definition of $\varphi$, we have that $M_{1}M_{2}\cdot P = P\cdot \varphi(M_{1}M_{2})$.
	
	\begin{rmk}
		\label{rmk:three-term-Fourier}
		In this context, the three term recurrence relation \eqref{eq:three_term_monic} can be written as
		$$xP = P\cdot x = L\cdot P, \qquad \text{where } \quad L=\delta + B(n) + C(n)\delta^{-1}.$$
		Therefore $x\in \mathcal{F}_R$, $L\in \mathcal{F}_L$ and $\varphi(L)=x$. For every polynomial $v \in \mathbb{C}[x]$, we have
		\begin{equation*}
			P\cdot v(x) =P\cdot v(\varphi(L))= v(L)\cdot P.
		\end{equation*}
	\end{rmk}
	
	On of the crucial results from \cite{Casper2} is the existence of an adjoint operation $\dagger$ in the Fourier algebras $\mathcal{F}_L(P)$ and $\mathcal{F}_R(P)$ as described in \cite[\S 3.1]{Casper2}. To define the adjoint operation in $\mathcal{F}_L(P)$, we initially observe that the algebra of difference operators $\mathcal{N}_N$ has a $\ast$-operation defined as follows:
	\begin{equation}\label{eq:star}
		\left( \sum_{j=-\ell}^k A_j(n) \, \delta^j \right)^\ast = \sum_{j=-\ell}^k A_j(n-j)^\ast \, \delta^{-j},
	\end{equation}
	where $A_j(n-j)^\ast$ is the conjugate transpose of $A_j(n-j)$. Now, the adjoint of $M\in \mathcal{N}_N$ is given by 
	\begin{equation}\label{eq:adjointM}
		M^\dagger = \mathcal{H}(n) M^\ast \mathcal{H}(n)^{-1},
	\end{equation}
	where $\mathcal{H}(n)$ is the squared norm which we view as an difference operator of order zero. The following holds:
	$$\langle (M\cdot P)(x,n),P(x,m)\rangle = \langle P(x,n),(M^\dagger \cdot P)(x,m)\rangle.$$
	
	In \cite[Corollary 3.8]{Casper2} the authors show that every differential operator $D\in  \mathcal{F}_R(P)$ has a unique adjoint $\D^\dagger\in  \mathcal{F}_R(P)$ with the property
	$$\langle P\cdot \D, Q \rangle = \langle P,Q\cdot \D^\dagger \rangle,$$
	for all $P,Q\in M_N(\mathbb{C})[x]$. Moreover, $\varphi(M^\dagger) = \varphi(M)^\dagger$ for all $M\in \mathcal{F}_L(P)$.

	\section{Semi-classical Laguerre type solutions}

	In the sequel, we consider the following two matrices $A,J\in M_{N}(\mathbb{R})$ which satisfy
	\begin{equation}\label{A J}
		J=\sum_{k=1}^{N} k E_{k,k} \qquad A =\sum_{k=1}^{N-1} a_{k} E_{k+1,k}.
	\end{equation}
	Notice that, it is straightforward to show that
	\begin{equation}\label{corchete A J}
		[J,A]=A \quad \text{and} \quad e^{xA}J e^{-xA}=J-Ax.
	\end{equation}
	Let us consider the following weight matrix supported on the interval $[0,\infty)$:
	\begin{equation}\label{W eq}
		W^{(\nu)}_{\phi}(x) = e^{Ax} T^{(\nu)}_{\phi}(x) e^{A^\ast x},\qquad T^{(\nu)}_{\phi}(x) =  e^{-\phi(x)} \sum_{k=1}^N \delta^{(\nu)}_k x^{\nu+k}  E_{k,k},
	\end{equation}
	where $\delta^{(\nu)}_k$ is a constant real number for $1 \leq k \leq N$, and $\phi$ be an analytic function on a neighborhood of the interval $[0,\infty)$.
	In the sequel, we assume that $W(x)P(x) = 0$ has vanishing limits at the endpoints of support for any matrix polynomial $P$.

	\begin{prop}\label{ort D, Ddag}
		Let $A,J\in M_{N}(\mathbb{C})$ as in \eqref{A J}. 
		Then, the first order differential operators
		$$\mathcal{D} = \partial_x x + x(A-1), \quad 
		\mathcal{D}^\dagger = -\partial_x x - (1+\nu+J)+x\phi'(x)-x,$$
		are mutually adjoint.
	\end{prop}
	\begin{proof}
		Let $P,Q \in M_N(\mathbb{C}[x])$. 
		In order to simplify notation, in the rest of the proof, we denote by $W(x):=W^{(\nu)}_{\phi}(x)$ and $T(x)=W^{(\nu)}_{\phi}(x)$
		\begin{eqnarray*}
			\langle P \cdot \mathcal{D}, Q \rangle &=& \int_{0}^{\infty} (P\cdot \mathcal{D}) W(x) Q^{\ast}(x) dx\\
			&=& \int_{0}^{\infty} \left( xP'(x)+xP(x)(A-1) \right) W(x) Q^{\ast}(x) dx.
		\end{eqnarray*}    
		Notice that, since $W(x)P(x)$ has vanishing limits at the endpoints $x=0$, and $x=\infty$, integration by parts implies that
		$$\int_{0}^{\infty} xP'(x)W(x)Q^{\ast}(x)dx = -\int_{0}^{\infty} P(x) \left(xW(x)Q^{\ast}(x) \right)'  dx.$$
		On the other hand, we have that
		$$\int_{0}^{\infty} xP(x)(A-1)W(x) Q^{\ast}(x) dx=\int_{0}^{\infty} P(x)W(x) \left( W^{-1}(x) x (A-1)W(x) \right)Q^{\ast} (x)  dx,$$
		by linearity we obtain that
		\begin{eqnarray*}
			\langle P \cdot \mathcal{D}, Q \rangle &=& - \int_{0}^{\infty} P(x) \left(xW(x)Q^{\ast}(x) \right)'  dx + \int_{0}^{\infty} P(x)W(x) \left( W^{-1}(x) x (A-1)W(x) \right)Q^{\ast} (x) dx\\
			&=& -\int_{0}^{\infty} P(x)\left(W(x)Q^{\ast}(x) +x W'(x) Q^{\ast}+xW(x)(Q^{\ast}(x))'\right) dx\\ 
			&&+ \int_{0}^{\infty} P(x)W(x) \left( W^{-1}(x) x (A-1)W(x) \right)Q^{\ast} (x) dx.
		\end{eqnarray*}
		Notice that since  $\left(Q \cdot (1+x\partial_{x}) \right)^{\ast}(x)=Q^{\ast}(x)+x(Q^{\ast}(x))'$, 
		we can put
		$$\int_{0}^{\infty} P(x)(W(x)Q^{\ast}(x) +W(x)x(Q^{\ast}(x))') dx=\int_{0}^{\infty} P(x)W(x)\left(Q \cdot (1+x\partial_{x}) \right)^{\ast}(x)  dx.$$
		On the other hand,
		$$\int_{0}^{\infty} P(x) x W'(x) Q^{\ast}(x) dx=\int_{0}^{\infty} P(x) W(x) xW^{-1}(x) W'(x) Q^{\ast}(x) dx.$$
		Hence, we obtain that
		\begin{eqnarray*}
			\langle P \cdot \mathcal{D}, Q \rangle &=& - \int_{0}^{\infty} P(x)W(x)\left(Q \cdot (1+x\partial_{x}) \right)^{\ast}(x)  dx - \int_{0}^{\infty} P(x) W(x) xW^{-1}(x) W'(x) Q^{\ast}(x) dx\\ 
			&&+ \int_{0}^{\infty} P(x)W(x) \left( W^{-1}(x) x (A-1)W(x) \right)Q^{\ast} (x) dx.
		\end{eqnarray*}
		By \eqref{W eq} we have that
		\begin{align*}
			W^{-1}(x) W'(x) &=e^{-A^{\ast} x} \left({T}^{-1}(x)\,A \,T(x)+T ^{-1}(x)T'(x)+ A^{\ast}\right) e^{A^{\ast} x},\\
			W^{-1}(x) x (A-1)W(x) &= e^{-A^{\ast}x} \left(x T^{-1}(x) \,A \,T(x)  - x\right)e^{A^{\ast}x}.
		\end{align*}
		Thus, we obtain 
		\begin{eqnarray*}
			\langle P \cdot \mathcal{D}, Q \rangle &=& - \int_{0}^{\infty} P(x)W(x)\left(Q \cdot (1+x\partial_{x}) \right)^{\ast}(x)  dx\\
			&&- \int_{0}^{\infty} P(x) W(x)  e^{-A^{\ast} x} \left(xT^{-1}(x) \,A \, T(x)+x T^{-1}(x)T'(x)+ x A^{\ast}\right) e^{A^{\ast} x} Q^{\ast}(x) dx\\ 
			&&+ \int_{0}^{\infty} P(x)W(x) e^{-A^{\ast}x} \left(x T^{-1}(x) \,A \,T(x)  - x\right)e^{A^{\ast}x} Q^{\ast}(x)dx\\
			&=& - \int_{0}^{\infty} P(x)W(x)\left(Q \cdot (1+x\partial_{x}) \right)^{\ast}(x)  dx\\
			&&- \int_{0}^{\infty} P(x) W(x)  e^{-A^{\ast} x} \left(x T^{-1}(x)T'(x)+ x A^{\ast}+x\right) e^{A^{\ast} x} Q^{\ast}(x) dx.
		\end{eqnarray*}
		By taking into account that $xT'(x)=T(x) (-x\phi'(x)+\nu+J)$, we obtain that
		\begin{eqnarray*}
			\langle  P \cdot \mathcal{D}, Q \rangle &=& - \int_{0}^{\infty} P(x)W(x)\left(Q \cdot (1+x\partial_{x}) \right)^{\ast}(x)  dx\\
			&-& \int_{0}^{\infty} P(x) W(x)  e^{-A^{\ast} x} ( -x\phi'(x)+\nu+J+xA^{\ast}+x) e^{A^{\ast}x} Q^{\ast}(x) dx.
		\end{eqnarray*}
		Notice that  the second expresion of the right hand of the above equality is
		$$e^{-A^{\ast} x} ( -x\phi'(x)+\nu+J+xA^{\ast}+x) e^{A^{\ast}x} = -x\phi'(x)+ \nu+ e^{-A^{\ast} x} J e^{A^{\ast}x}+xA^{\ast}+x.$$
		On the other hand, the equation $e^{xA}J e^{-xA}=J-Ax$
		implies that $e^{-A^{\ast} x} J e^{A^{\ast}x}=J-A^{\ast}x$. Hence, we obtain that
		
		\begin{eqnarray*}
			\langle P \cdot \mathcal{D}, Q \rangle  &=& - \int_{0}^{\infty} P(x)W(x)\left(Q \cdot (1+x\partial_{x}) \right)^{\ast}(x)  dx\\
			&&- \int_{0}^{\infty} P(x) W(x)   (x -x\phi'(x)+(\nu+J-A^{\ast}x+xA^{\ast}) Q^{\ast}(x) dx\\ 
			&=& \int_{0}^{\infty} P(x)W(x)\left(Q \cdot -(x\partial_{x}+x-x\phi'(x)+(\nu+J+1) \right)^{\ast}(x)  dx\\
			&=&\langle P,Q\cdot \mathcal{D}^{\dagger}\rangle.
		\end{eqnarray*}
		Therefore, the operators $\mathcal{D}$ and $\mathcal{D}^\dagger$ are mutually adjoint, as asserted.
	\end{proof}
	
	By the above theorem, since $\mathcal{D} = \partial_x x + x(A-1)$ and
	$\mathcal{D}^\dagger = -\partial_x x - (1+\nu+J)+x\phi'(x)-x$, then we obtain that
	\begin{equation}\label{D-Ddaga eqn}
		\mathcal{D}^\dagger = -\mathcal{D} + (Ax-J) -(1+\nu) + x\phi'(x)-2x.    
	\end{equation}
	
	\begin{lem}\label{AJ lemma}
		Let $A,J\in M_{N}(\mathbb{C})$ be matrices as in \eqref{A J}. 
		Then, $\mathcal{C}=Ax-J$  is a symmetric operator respect to the weight $W:=W^{(\nu)}_{\phi}$ as in \eqref{W eq}. 
		Moreover, if $P(x,n)$'s are monic MVOPs associated with the weight $W$, such that
		$$P\cdot\mathcal{C}=M_{\mathcal{C}}\cdot P, \quad \text{then} \quad M_{\mathcal{C}}=\sum_{i=-1}^{1} U_j(n)\delta^j,$$
		with $$U_1(n):=A,\qquad U_0(n):=X(n)A-AX(n+1)-J,$$ $$U_{-1}(n):= Y(n)A-AY(n+1) + [J,X(n)] + (AX(n+1)-X(n)A)X(n),$$ 
		where $X(n)$ and $Y(n)$ are the coefficients of the $(n-1)$-term and $(n-2)$-term of $P(x,n)$ respectively.
	\end{lem}
	\begin{proof}
		In the same manner as Proposition \ref{ort D, Ddag}, it can be shown that $\mathcal{C}$ is an symmetric operator respect with $W$.
		
		If we put $M_{\mathcal{C}}=\sum U_{j}(n)\delta^j$ such that $P\cdot\mathcal{C}=M_{\mathcal{C}}\cdot P$.
		By taking into account that $\mathcal{C}$  increases the degree of any polynomial in $1$, we obtain that $U_{j}(n)=0$ for $j>1$. On the other hand, since $\mathcal{C}$ is an symmetric operator respect to $W$, we have that $M_{\mathcal{C}}=M^{\dagger}_{\mathcal{C}}$ and so $U_{j}(n)=0$ for $j<-1$.
		
		The formulas for $U_{-1}(n),U_{0}(n)$ and $U_{1}(n)$ can be obtained by direct computation from $P\cdot\mathcal{C}=M_{\mathcal{C}}\cdot P$.
	\end{proof}
	
	\begin{rmk}
		Since $\mathcal{C}=Ax-J=(Ax-J)^{\dagger}$, then $M_{\mathcal{C}}=M_{\mathcal{C}}^{\dagger}$.
		Thus, from equations \eqref{eq:adjointM} we have that 
		\begin{align*}
			U_1(n) &=\mathcal{H}(n) U_{-1}(n+1)^{\ast}\mathcal{H}(n+1)^{-1},\\ 
			U_0(n) &=\mathcal{H}(n)U_0(n)^{\ast}\mathcal{H}(n)^{-1},
		\end{align*}
		and so we obtain that
		$$A=\mathcal{H}(n)\Big(Y(n+1)A-AY(n+2)+[J,X(n+1)]+(AX(n+2)-X(n+1)A)X(n+1)\Big)^{\ast} \mathcal{H}(n+1)^{-1},$$
		$$X(n)A-AX(n+1)-J=\mathcal{H}(n) \Big(X(n)A-AX(n+1)-J\Big)^{\ast} \mathcal{H}(n)^{-1}.$$ 
	\end{rmk}
	
	\begin{thm}\label{ Accion general D M}
		Let $W:=W^{(\nu)}_{\phi}$ be a matrix weight as in \eqref{W eq}, 
		with monic MVOPs $P(x,n)$ such that
		$$\mathcal{D}= \partial_{x} x + (A-1)x, \qquad \mathcal{D}^{\dagger}= -\mathcal{D} + \mathcal{C} + v'(x)$$
		for some polynomial $v(x)$ of degree $k$ and $\mathcal{C}=Ax-J$. 
		If $X(n)$ and $Y(n)$ are the coefficients of the $(n-1)$-term and $(n-2)$-term of $P(x,n)$ respectively. 
		Then, the monic polynomials $P(x,n)$ satisfy the following relation
		\begin{equation}\label{PDM}
			P\cdot\mathcal{D}=M\cdot P, \qquad M= \sum_{j=-k+1}^{1} A_{j}(n)\delta^j    
		\end{equation}
		with 
		{\small $$A_1(n)=A-1,\quad A_0(n)=n+X(n)A - A X(n+1)-B(n),$$ 
			$$A_{-1}(n)=(n-1)X(n)+Y(n)(A-1)-(A-1)Y(n+1)-A_{0}(n)X(n),$$
			$$A_j(n)=(v'(L))_j(n), \quad -k+1 <j< -1$$}
		where $B(n)$ is given by \eqref{B C}.
	\end{thm}
	\begin{proof}
		Clearly, the formulas for $A_{j}(n)$ with $j=-1,0,1$ can be derived from the equalities in $\eqref{PDM}$ by using the definition of $\mathcal{D}$.
		
		For $j<-1$, we have that
		\begin{eqnarray*}
			A_{j}(n) &=& \langle P\cdot \mathcal{D},\delta^j \cdot P \rangle \mathcal{H}(n-j)^{-1}=\langle P, \delta^j \cdot P \cdot \mathcal{D}^{\dagger}\rangle \mathcal{H}(n-j)^{-1}\\
			& =& \langle P, \delta^j \cdot P \cdot v'(x)\rangle\mathcal{H}(n-j)^{-1}=\langle P\cdot v'(x), \delta^j \cdot P\rangle\mathcal{H}(n-j)^{-1}\\
			&= & \langle v'(L)\cdot P, \delta^{j}\cdot P \rangle\mathcal{H}(n-j)^{-1},
		\end{eqnarray*}
		where we have used that $\langle P, \delta^j \cdot P\cdot \mathcal{D}\rangle$ and $\langle P, \delta^j \cdot P\cdot \mathcal{C}\rangle$ are both zero for $j<-1$ in the third equality, and the fact that $v'(x)$ is a scalar funtion in the fourth one. Then, we have that 
		$$A_{j}(n)= (v'(L))_j(n) \quad \text{for $j<-1$}.$$
		To complete the proof, notice that $(v'(L))_{j}(n)=0$ for $j\leq -k$.
	\end{proof}
	
	As a direct consequence, we obtain the following corollary.
	\begin{cor}\label{coro v grado 1}
		In the same hypothesis as in Theorem \ref{ Accion general D M}.
		If the polynomial $v$ has degree $1$, then the discrete operator $M$ associated  with $\mathcal{D}$ satisfies
		$$M=  A_{0}(n)+(A-1)\delta$$
		with $A_{0}(n)$ as in Theorem \ref{ Accion general D M}. Moreover, in this case we have that
		{\small\begin{equation}\label{fla A-1n}
				(n-1)X(n)+Y(n)(A-1)-(A-1)Y(n+1)-\Big( n+X(n)A - A X(n+1)-B(n)\Big)X(n) = 0.
		\end{equation}}
	\end{cor}
	\begin{proof}
		By Theorem \ref{ Accion general D M}, 
		since $\deg (v) < 2$ we obtain that $A_{j}(n)=0$ for all $j<-1$.
		On the other hand, notice that since $v$ has degree $1$, 
		then $v'(x)$ is a constant function and so the operator
		$\mathcal{D}^{\dagger}$ does not increase degrees. 
		This implies that $A_{1}^{\dagger}=0$ and therefore 
		$$A_{-1}(n)=(n-1)X(n)+Y(n)(A-1)-(A-1)Y(n+1)-A_{0}(n)X(n) = 0,$$ 
		the formula for $A_{1}(n)$ and $A_{0}(n)$ are consequence of Theorem \ref{ Accion general D M}.
	\end{proof}
	\section{Lie algebras associated to orthogonal polynomials}
	
	In this section, we solve a particular case of the problem proposed by Ismail in \cite[Problem 24.5.2]{ISMAIL2019235}, as described in the introduction. For this purpose, we study the structure of a Lie algebra related with the operators $\mathcal{D}$ and $\mathcal{D}^{\dagger}$. 
	
	Recall that if $\mathfrak{g}$ is a finite dimensional Lie algebra, and if {$\mathfrak{g}^{j}$} and $\mathfrak{g}_j$ denote the following recursions
	$$\mathfrak{g}^0=\mathfrak{g}_0=\mathfrak{g}, \quad \mathfrak{g}^{j+1}=[\mathfrak{g}^j,\mathfrak{g}^{j}] \quad \text{and } \quad \mathfrak{g}_{j+1}=[\mathfrak{g},\mathfrak{g}_{j}],$$
	then $\mathfrak{g}$ is called solvable (nilpotent) if $\mathfrak{g}^j=0$ for some $j$ (if $\mathfrak{g}_j=0$ for some $j$).
	Clearly, any nilpotent Lie algebra is solvable.
	The radical (nilradical) of $\mathfrak{g}$ is its maximal solvable ideal (maximal nilpotent ideal) of $\mathfrak{g}$.
	We will denote by $\mathrm{Rad}(\mathfrak{g})$ and $\mathrm{Nil}(\mathfrak{g})$ to the radical and nilradical of $\mathfrak{g}$, respectively.

	\subsection*{Lie algebra generated by $\mathcal{D}$ and $\mathcal{D}^{\dagger}$}

	\begin{lem}\label{bracket D x xphi}
		Let $A,J\in M_{N}(\mathbb{C})$ as in \eqref{A J} and let $\phi$ an entire function over $\mathbb{C}$, let us consider the operators
		$$\mathcal{D} = \partial_x x + x(A-1), \quad 
		\mathcal{D}^\dagger = -\partial_x x - (1+\nu+J)+x\phi'(x)-x.$$
		If $x$ and $x^{j}\phi^{(j)}(x)$  act over matrix valued polynomials by right constant multiplication, then we have that 
		$$[\mathcal{D},x] = -x,\qquad [\mathcal{D}^\dagger,x]=x,\qquad [\mathcal{D},\mathcal{D}^\dagger] = -x^2\phi^{(2)}(x)+(2-\phi'(x))x,$$
		$$[\mathcal{D},\phi^{(j)}(x)x^j] =  -(jx^{j}\phi^{(j)}(x)+x^{j+1}\phi^{(j+1)}(x)) = -[\mathcal{D}^{\dagger},\phi^{(j)}(x)x] \quad \text{for all $j\ge 1$}.$$	
	\end{lem}
	
	In the sequel, given $\phi$ an entire function over $\mathbb{C}$, we denote by
	\begin{equation}\label{associated lie algebra}
		\mathfrak{g}_{\phi}= \langle 1,\mathcal{D}, \mathcal{D}^{\dagger}, x, x\phi'(x), x^{2} \phi^{(2)}(x),\ldots \rangle
	\end{equation}
	with bracket as above. We are interested in the case that this Lie algebra is finite dimensional. The following proposition states that this happens if and only if $\phi$ is a polynomial. We will need the following notation, given $\phi$ a polynomial over $\mathbb{C}$ with $\ell$ non-zero coefficients
	\begin{equation}\label{def k}
		k=\begin{cases}
			\ell+2 & \text{if } \phi'(0)=\phi(0)=0,\\
			\ell+1 & \text{if } \phi'(0)=0, \phi(0)\neq 0,\\
			\ell+1 & \text{if } \phi(0)=0, \phi'(0)\neq 0,\\
			\ell & \text{if } \phi(0)\neq 0, \phi'(0)\neq 0,
		\end{cases}
	\end{equation}	
	
	\begin{prop}\label{Prop dimension}
		Let $\phi$ an analitic function over $\mathbb{C}$ and let $\mathfrak{g}:=\mathfrak{g}_{\phi}$ its associated Lie algebra as in \eqref{associated lie algebra}. 
		Then, we have that
		$\dim (\mathfrak{g})$ is finite if and only if $\phi$ is a polynomial. 
		In such case, if $k$ is as in \eqref{def k} then
		$$\dim (\mathfrak{g})= k+2.$$
	\end{prop}
	\begin{proof}
		Clearly, if $\phi$ is a polynomial, then the dimension of $\mathfrak{g}$ is finite, since if $n$ is the degree of $\phi$ then $\phi^{(m)}(x)=0$ for all $m>n$. 
		
		Conversely, assume now that $\phi$ is not a polynomial. 
		Since $\phi$ is analytic, we can express $\phi$ as follow
		$$\phi(x)= \sum_{i=0}^{\infty} a_{i}x^{i}.$$
		This implies that 
		\begin{equation}\label{eq xjphi}
			x^{j}\phi^{(j)}(x)= \sum_{i=0}^{\infty} b_{i,j} x^{i}, \quad \text{with} \quad b_{i,j}=
			\begin{cases}
				\tbinom{i}{j}j! \,a_i & \text{ if $j\leq i$}, \\   
				0  & \text{if $i<j$}. 
			\end{cases}  
		\end{equation}
		In particular, if $a_i=0$ then $b_{i,j}=0$ for all $j \in \mathbb{N}$.
		
		Let $\{i_t\}_{t\in \mathbb{N}}$ be the sequence 
		of non-zero coefficients indices of $\phi$, that is  $i_{t}<i_{t+1}$ for all $t \in \mathbb{N}$, such that $a_i \neq 0$ if and only if $i= i_{t}$ for some $t \in \mathbb{N}$.

		\medskip
		
		\noindent \textit{Claim}: The vector space $\langle x^{i_1}\phi^{(i_1)}(x),\ldots, x^{i_{\ell}}\phi^{(i_{\ell})}(x)\rangle$ has dimension $\ell$.
		
		Let $c_{1},\ldots,c_{\ell} \in \mathbb{C}$ such that 
		$$c_1x^{i_1}\phi^{(i_1)}(x)+\cdots+ c_{\ell}x^{i_\ell}\phi^{(i_\ell)}(x)=0,$$
		this induces the following system of equations
		$$\sum_{t=1}^{h} c_{t} \tbinom{i_h}{i_t} i_{h}! a_{i_t}=0 \quad \text{for $h=1,\ldots,\ell$}.$$
		By taking into account that $a_{i_1}\neq 0$, 
		the equation for $h=1$ implies that $c_1=0$. 
		In the same way, since $c_1=0$, the equation for $h=2$ implies that
		$c_2 a_{i_2} i_2!=0$ and so $c_2=0$ since $a_{i_2}\neq 0$. Inductively,
		if $c_1=c_2=\ldots=c_{\ell-1}=0$, then the equation for $h={\ell}$ implies that
		$c_{\ell} a_{i_{\ell}} i_{\ell}!=0$ and so $c_{\ell}=0$ since $a_{i_{\ell}}\neq 0$, 
		hence we obtain that $c_{h}=0$ for all $h\in \{1,\ldots,\ell\}$.
		
		By the claim, the space $\mathfrak{g}$ has subspaces of all of the possible dimensions and so is non-finite dimensional, as asserted.
		
		Now, assume that $\phi$ is a polynomial of degree $n$, in the same notation as above,
		by \eqref{eq xjphi} we have that 
		$$\langle x\phi'(x),x^{2}\phi^{(2)}(x),\ldots, x^{n}\phi^{(n)}(x)\rangle\subseteq \langle x^{i_1},\ldots,x^{i_\ell}\rangle.$$
		The claim and the above statement imply that $\langle x\phi'(x),x^{2}\phi^{(2)}(x),\ldots, x^{n}\phi^{(n)}(x)\rangle$ has dimension $\ell$. 
		
		Finally, the last assertion follows from the fact that
		$\mathcal{D},\mathcal{D}^{\dagger}$ are linearly independent respect to
		$$\langle 1,x,x\phi'(x),x^{2}\phi^{(2)}(x),\ldots, x^{n}\phi^{(n)}(x)\rangle$$ 
		and this vector space has dimension $k$, with $k$ as in the statement.
	\end{proof}
	
	\begin{rmk}
		The above proposition solves the problem proposed by Ismail in \cite[Problem 24.5.2]{ISMAIL2019235} for the case $B_n=-1$ for all $n$ natural number. In the notation of \cite{ISMAIL2019235}, the differential operators $\mathcal{D}$ and $\mathcal{D}^{\dagger}$ correspond to $\mathcal{D}=x L_{1,n}$ and $\mathcal{D}^{\dagger}=x L_{2,n} +(1+\nu)$. Then, the algebra generated by $\{\mathcal{D}, \mathcal{D}^{\dagger},1\}$ is isomorphic to the algebra generated by $\{xL_{1,n}, xL_{2,n},1\}$.
	\end{rmk}

	\begin{rmk}\label{rem phi espacio}
		By the proof of the above theorem,
		if $\phi(x)=a_0+a_1 x +\ldots+ a_{n}x^n$ is a polynomial of degree $n$ with $\ell$ non-zero coefficients.
		If $\{i_1,\ldots,i_{\ell} \}\subseteq \{0,\ldots,n\}$ is the set of indices such that $a_{i_j}\neq 0$.
		then we have that
		\begin{equation}
			\langle x\phi'(x),x^{2}\phi^{(2)}(x),\ldots, x^{n}\phi^{(n)}(x)\rangle=\langle x^{i_1},\ldots,x^{i_\ell}\rangle.
		\end{equation}
	\end{rmk}

	\begin{example}
		Let $\phi_{1}(x)=x^3$ and $\phi_{2}(x)=x^3+x^2$, by the above theorem the associated Lie algebras
		$\mathfrak{g}_{\phi_1}$ and $\mathfrak{g}_{\phi_2}$ have the dimensions $5$ and $6$, respectively. Then, the algebras
		$\mathfrak{g}_{\phi_1}$ and $\mathfrak{g}_{\phi_2}$ are non-isomorphic. 
	\end{example}

	\begin{lem}\label{lema elemento central}
		The element $z= \mathcal{D}+\mathcal{D^{\dagger}} +2x -x\phi'(x)$ is a symmetric differential operator which
		belongs to the center of the Lie algebra $\mathfrak{g}$.
	\end{lem}
	\begin{proof}
		It is follows immediately from the definition of the bracket of $\mathfrak{g}_{\phi}$.
	\end{proof}
	
	\begin{rmk}
		The central element that we found in the above lemma, it is related with the symmetric operator $\mathcal{C}$ that was considered in Lemma \ref{AJ lemma},
		we will see this in the following section.
	\end{rmk}
	
	In the sequel, given a polynomial $\phi(x)=a_{0}+a_1 x+ \cdots +a_n x^n$  of degree $n\geq 2$ 
	with $\ell$ non-zero coefficients and $k$ as in \eqref{def k}, let us consider the following notations. 
	\begin{equation}\label{def Iphi}
		I_{\phi}=\{i \in \{2,\ldots,n\}: a_i\neq 0\}=\{ j_1,\ldots,j_{k-2}\}=
		\begin{cases}
			\{i_3,\ldots,i_{\ell}\} & \text{if $a_0\neq 0$, $a_1\neq 0$,} \\
			\{i_2,\ldots,i_{\ell}\} & \text{if $a_0=0$ and $a_1\neq 0$,} \\
			\{i_2,\ldots,i_{\ell}\} & \text{if $a_0\neq 0$ and $a_1= 0$,} \\
			\{i_1,\ldots,i_{\ell}\} & \text{if $a_0=0$ and $a_1=0$,} \\
		\end{cases}
	\end{equation}
	with $j_t< j_{t+1}$ and $i_t<i_{t+1}$.
	
	\begin{thm}\label{Teo estructura}
		Let $\phi(x)=a_{0}+a_1 x+ \cdots +a_n x^n$ be a polynomial of degree $n\geq 2$ 
		with $\ell$ non-zero coefficients and $k$ as in \eqref{def k}. 
		If $\mathfrak{g}:=\mathfrak{g}_{\phi}$ is the associated Lie algebra of $\phi$ as in \eqref{associated lie algebra},
		then we have that
		$$\mathfrak{g}\cong \mathbb{C}^2 \oplus \mathfrak{h}$$
		where $\mathfrak{h}$  is a solvable Lie algebra of dimension $k$, with an abelian nilradical of dimension $k-1$. More precisely,
		if $I_{\phi}$ is as in \eqref{def Iphi}	then
		$$\mathfrak{h} \cong \langle E \rangle \ltimes \langle E_1, \ldots E_{k-1}\rangle $$
		where $\langle E_1, \ldots E_{k-1}\rangle$ is abelian and the rest of the brackets satisfy
		\begin{equation}\label{bracket E's}
			[E,E_1]=E_1 \quad \text{and} \quad [E,E_t]= j_{t-1} E_t \quad \text{for $t=2,\ldots,k-1$}.
		\end{equation}
	\end{thm}
	\begin{proof}
		By Lemma \ref{lema elemento central}, the element 
		$z= \mathcal{D}+\mathcal{D^{\dagger}} +2x -x\phi'(x)$ belongs to the center of $\mathfrak{g}$, 
		and so we obtain an element in the center which does not belong to $\langle 1\rangle$, 
		thus if 
		\begin{equation*}\label{def h}
			\mathfrak{h}=\langle \mathcal{D}, x,x\phi'(x),x^{2}\phi^{(2)}(x),\ldots, x^{n}\phi^{(n)}(x) \rangle   
		\end{equation*}
		then we obtain that
		$$\mathfrak{g}\cong \mathbb{C}^2 \oplus \mathfrak{h},$$
		with $\mathfrak{h}$  a Lie algebra of dimension $k$.
		
		Thus, it is enough to show that $\mathfrak{h}$ is solvable with nilradical of dimension $k-1$. 
		Let us consider
		$$\mathfrak{k}=\langle x,x\phi'(x),x^{2}\phi^{(2)}(x),\ldots, x^{n}\phi^{(n)}(x)\rangle,$$
		by definition of the bracket and by taking into account that 
		\begin{equation}\label{corchetes x phi}
			[x,x^{l} \phi^{(l)}(x)]=[x^i \phi^{(i)}(x),x^j \phi^{(j)}(x)]=0\qquad \text{for all $i\neq j$,}    
		\end{equation}
		we obtain that $[\mathfrak{h},\mathfrak{h}]\subseteq \mathfrak{k}$.
		Hence, by \eqref{corchetes x phi} we obtain that $\Big[[\mathfrak{h},\mathfrak{h}],[\mathfrak{h},\mathfrak{h}]\Big]=0,$
		and so $\mathfrak{h}$ is solvable.
		Finally, notice that
		$\mathfrak{k}$ is an abelian ideal of $\mathfrak{h}$ of dimension 
		$$\dim (\mathfrak{k})=\dim (\mathfrak{h})-1=k-1, $$
		this implies that $\mathfrak{k}$ is the nilradical of $\mathfrak{h}$, as desired.
		
		In the same manner as in Remark \ref{rem phi espacio} we have that
		$$\mathfrak{h}=\langle \mathcal{D}, x, x^{j_1},\ldots, x^{j_{k-2}} \rangle,$$ 
		in this case $\langle x, x^{j_1},\ldots, x^{j_{k-2}}\rangle$ is an abelian subalgebra of dimension $k-1$. 
		It is enough to see the brackets $[\mathcal{D},x^{j_t}]$ and $[\mathcal{D},x]$, in this case we obtain that
		$$[\mathcal{D},x]=-x \qquad \text{and} \qquad [\mathcal{D},x^{j_t}]= -j_{t} x^{j_t},$$
		so we obtain that $\langle x, x^{j_1},\ldots, x^{j_{k-2}}\rangle$ is an abelian ideal. Finally, 
		if we consider the following correspondence
		$$\mathcal{D}\longmapsto -E, \quad x\longmapsto E_1, \quad x^{j_i}\longmapsto E_{i+1} \quad \text{for $i=1,\ldots,k-2$},$$
		is an Lie Algebra isomorphism between $\mathfrak{h}$ and $\langle E \rangle \ltimes \langle E_1, \ldots E_{k-1}\rangle$ with brackets given as in \eqref{bracket E's}, as asserted.
	\end{proof}

	In the sequel we are going to study the structure of the solvable Lie algebra $\mathfrak{h}_{\phi}$.
	In general, a Lie algebra $\mathfrak{h'}$ with an abelian ideal of codimension $1$ is called almost abelian.
	This kind of algebra was studied by V.V.\@ Gorbatsevich in \cite{gorbatsevich1998level}. 
	The author asserted that in general this kind of algebra it decomposes as
	$$\mathbb{C} \ltimes_{\psi} \mathbb{C}^{k-1},$$
	this semidirect product gives a linear transformation $\psi: \mathbb{C} \rightarrow gl_{k-1}(\mathbb{C})$,
	moreover he asserted that the structure of this kind of algebras it determines by the matrix $\Psi=\psi(1)$. More precisely,
	\begin{equation}\label{conformally similar}
		\mathbb{C} \ltimes_{\psi} \mathbb{C}^{k-1}\cong \mathbb{C} \ltimes_{\psi'} \mathbb{C}^{k-1} \Longleftrightarrow \quad \text{$\Psi$ and $\Psi'$ are conformally similar},
	\end{equation}
	recall that $\Psi$ and $\Psi'$ are conformally similar if and only if there exist a matrix $P\in Gl_{k-1}(\mathbb{C})$ and a non-zero complex number $\lambda \in \mathbb{C}\smallsetminus \{0\}$ such that $\Psi= \lambda P \Psi' P^{-1}$.
	
	\begin{lem}\label{Lema conformally}
		Let $k$ be an integer greater than $1$ and let $1<j_1<\ldots<j_{k-2}$ and $1<j'_1<\ldots< j'_{k-2}$ be two sequences of positive integers.
		Let us consider 
		$$\Psi = \mathrm{diag}(1, j_1,\ldots,j_{k-2}) \quad  \text{and} \quad \Psi' = \mathrm{diag}(1, j'_1,\ldots,j'_{k-2}).$$ 
		Then, $\Psi$ and $\Psi'$ are conformally similar if and only if $\Psi=\Psi'$.
	\end{lem}	
	\begin{proof}
		Clearly, if $\Psi=\Psi'$ then they are conformally similar trivially.
		
		Now, assume that $\Psi$ and $\Psi'$ are conformally similar, so there exist $\lambda\in \mathbb{C}^*$ and $P\in GL_{k-1}(\mathbb{C})$ such that
		$$\Psi' = \lambda P \Psi P^{-1},$$ 
		by similarity we obtain the following spectral relationship
		$$\mathrm{Spec}(\Psi')= \lambda\cdot \mathrm{Spec}(\Psi),$$
		i.e.\@ all of the eigenvalues of $\Psi'$ can be obtained from the eigenvalues of $\Psi$ by multiplication by $\lambda$.
		Since $\Psi$ and $\Psi'$ are both diagonal, we have that 
		$$\mathrm{Spec}(\Psi)=\{1,j_1,\ldots,j_{k-2}\}\quad \text{ and} \quad \mathrm{Spec}(\Psi')=\{1, j'_1,\ldots,j'_{k-2}\},$$ 
		since $1$ belong to both spectra and the rest of the eigenvalues of $\Psi$ and $\Psi'$ are greater than $1$, we obtain that $\lambda=1$ necessarily. 
		Hence, we obtain that
		$$\mathrm{Spec}(\Psi')= \mathrm{Spec}(\Psi).$$
		Therefore, $\Psi=\Psi'$ as asserted.
	\end{proof}
	
	We are in position to give the following theorem, which says when the Lie algebra associated to two different polynomials (as in \eqref{associated lie algebra}) are isomorphic.
	\begin{thm}\label{Teo lie algebra isom}
		Let $\phi_1(x),\phi_2(x)$ be polynomials over $\mathbb{C}$ of degree greater or equal than $2$.
		Let $\mathfrak{h}_{\phi_1}$ and $\mathfrak{h}_{\phi_2}$ be its associated solvable Lie algebras given as in Theorem \ref{Teo estructura}.
		Then, 
		$$\mathfrak{h}_{\phi_1}\cong \mathfrak{h}_{\phi_2} \Longleftrightarrow I_{\phi_1} = I_{\phi_2},$$
		where $I_{\phi_1}, I_{\phi_2}$ are as in \eqref{def Iphi}.
		Moreover, we have that
		$$\mathfrak{g}_{\phi_1}\cong \mathfrak{g}_{\phi_2} \Longleftrightarrow I_{\phi_1} = I_{\phi_2},$$
		where $\mathfrak{g}_{\phi_1}$ and $\mathfrak{g}_{\phi_2}$ are the associated Lie algebras of $\phi_1$ and $\phi_2$, respectively.
	\end{thm}
	\begin{proof}
		From Theorem \ref{Teo estructura}, we have that $\mathfrak{g}_{\phi_1}\cong \mathfrak{g}_{\phi_2}$ if and only if $\mathfrak{h}_{\phi_1}\cong \mathfrak{h}_{\phi_2}$.
		So, it is enough to see that
		$$\mathfrak{h}_{\phi_1}\cong \mathfrak{h}_{\phi_2} \Longleftrightarrow I_{\phi_1} = I_{\phi_2}.$$
		Now by \eqref{conformally similar}, it enough to see that the associated matrices $\Phi_1$ and $\Phi_2$ are conformally similar. 
		By Theorem \ref{Teo estructura} and \eqref{conformally similar} we have that $\Phi_1$ and $\Phi_2$ are conformally similar to 
		$$\Psi = \mathrm{diag}(1, j_1,\ldots,j_{k-2}) \quad  \text{and} \quad \Psi' = \mathrm{diag}(1, j'_1,\ldots,j'_{k-2}) \quad \text{respectively,}$$
		where  $\{j_1,\ldots,j_{k-2}\} =I_{\phi_1}$ and $\{j'_1,\ldots,j'_{k-2}\}=I_{\phi_2}$.
		Finally, by Lemma \ref{Lema conformally}, we obtain that $\Psi$ and $\Psi'$ are conformally similar if and only if 
		$\Psi=\Psi'$ which is equivalent to say that $I_{\phi_1}=I_{\phi_2}$. 
		Therefore $\mathfrak{h}_{\phi_1}\cong \mathfrak{h}_{\phi_2}$ if and only if $I_{\phi_1}= I_{\phi_2}$ as asserted.	
	\end{proof}
	
	As a direct consequence, we obtain the following.
	\begin{cor}
		Let $\phi_1(x),\phi_2(x)$ be polynomials over $\mathbb{C}$ with degree greater than $2$. 
		Then, we have the following cases:
		\begin{enumerate}
			\item If $\deg \phi_1=\deg \phi_2=2$, then $\mathfrak{g}_{\phi_1}\cong \mathfrak{g}_{\phi_2}$.
			\item If $\deg \phi_1\neq \deg \phi_2$, then $\mathfrak{g}_{\phi_1}\not\cong \mathfrak{g}_{\phi_2}$.
		\end{enumerate}	
	\end{cor}
	
	\begin{rmk}
		Notice that if we consider $\phi_m(x)=x^{m}+ax$ with $m\ge 2$, then $\dim(\mathfrak{h}_{\phi_m})=3$. 
		The structure of solvable Lie algebras of dimension $3$ 
		was studied by J.\@~Patera and H.\@~Zassenhaus in \cite{PATERA19901}. 
		In page 4, the authors 
		define the Lie algebra 
		$L_{3,6}=\langle a_1,a_2,a_3\rangle$ with brackets
		$$[a_1,a_2]= a_3\qquad [a_1,a_3]= a_3-\alpha \cdot a_2$$
		with parameter $\alpha$ satisfying $\alpha \neq 0$ and $1-4\alpha\neq 0$,
		this parameter $\alpha$ is in one-to-one correspondence with isomorphism classes of this kind of algebras.
		Its associated matrix $\psi_{\alpha}(1)$ is 	
		\begin{equation*}
			\psi_{\alpha}(1)=
			\begin{pmatrix}
				0 & -\alpha\\
				1 & 1
			\end{pmatrix}
		\end{equation*}
		The eigenvalues of $\psi_{\alpha}(1)$ are 
		$$\lambda_0=\tfrac{1-\sqrt{1-4\alpha}}{2}\quad \text{and} \quad \lambda_1=\tfrac{1+\sqrt{1-4\alpha}}{2}.$$
		On the other hand, if we consider $\phi_m(x)=x^{m}+ax$ with $m\ge 2$, 
		then $\mathfrak{h}_{\phi_m}$ has dimension $3$ and
		its associated matrix $\Psi_m$ is conformally similar to 
		$$ \begin{pmatrix}
			1 & 0\\
			0 & m
		\end{pmatrix}$$ 
		thus, $\Psi_m$ is conformally similar to $\psi_{\alpha}(1)$ if and only if 
		$$\lambda_0=r, \qquad \lambda_1=rm \qquad \text{for some complex number $r$},$$
		since diagonalizable matrices are similar if and only if its spectrum are equal. 
		This system of equation has a solution $r=\tfrac{1}{m+1}$ and $\alpha= \tfrac{m}{(m+1)^{2}}$.
		Therefore $\mathfrak{h}_{\phi_m} \cong L_{3,6}^{\alpha}$ with $\alpha=\tfrac{m}{(m+1)^{2}}$.
		
	\end{rmk}
	
	\section{Laguerre type solutions}
	Let $A\in M_N(\mathbb{C})$ be a constant matrix and let $\nu\in \mathbb{R}$ such that $\nu>0$. 
	In this section and the sequel, we are going to 
	consider the weight matrix  $W$ given by
	\begin{equation}\label{def: Wnu}
		W^{(\nu)}(x) = e^{Ax} T^{(\nu)}(x) e^{A^\ast x},\qquad T^{(\nu)}(x) =  e^{-x} \sum_{k=1}^N \delta^{(\nu)}_k x^{\nu+k}  E_{k,k}.	\end{equation}
	If we denote $L(x)=e^{Ax}$, then 
	\begin{equation}\label{rel W T}
		W^{(\nu)}(x)=L(x) T^{(\nu)}(x) L^{*}(x). 
	\end{equation}
	Recall that if $A,J$ are as in \eqref{A J}, the equation \eqref{corchete A J} 
	in terms of $L$ says that
	\begin{equation}\label{rel L J}
		L(x)J L^{-1}(x)=J- Ax.
	\end{equation}
	In this case, we obtain the same kind of weight that was consider in section 3, with $\phi(x)=x$.
	
	\begin{prop}\label{prop cal D y D dag}
		The first order differential operators
		$$\mathcal{D} = \partial_x x + x(A-1), \quad 
		\mathcal{D}^\dagger = -\partial_x x - (1+\nu +J),$$
		are mutually adjoint and 
		satisfy
		\begin{align*}
			M = \psi^{-1}(\mathcal{D}) = & (A-1)\delta-(n+1+\nu)-\mathcal{H}(n)J\mathcal{H}^{-1}(n),\\
			M^\dagger = \psi^{-1}(\mathcal{D}^{\dagger}) =& -(n+\nu+J+1) + \mathcal{H}(n)(A-1)^\ast \mathcal{H}^{-1}(n-1) \delta^{-1}.
		\end{align*}
	\end{prop}
	\begin{proof}
		The operators $\mathcal{D}$ and $\mathcal{D^{\dagger}}$ are mutually adjoint by taking $\phi(x)=x$ in Proposition \ref{ort D, Ddag}. 
		
		On the other hand, since $v(x)$ has degree $1$ in this case, by Corollary \ref{coro v grado 1} 
		we obtain that $A_{j}(n)=0$ for $j\le -1$ and $A_{1}(n)=A-1$.
		
		Finally, the formula for $A^{\dagger}_{0}(n)$ can be obtained directly from the relation $P\cdot \mathcal{D}^{\dagger}=M_{\mathcal{D}^{\dagger}}\cdot P$ 
		and the term $A_{0}(n)$ can be obtained from equation \eqref{eq:adjointM}.
	\end{proof}
	
	As a direct consequence of the above proposition and  Corollary \ref{coro v grado 1} we obtain the following result.
	\begin{cor}\label{Coro formula A0n}
		Let $A,J\in M_{N}(\mathbb{C})$ as in \eqref{A J} and let $W:=W^{(\nu)}_{\phi}$ be a matrix weight as in \eqref{W eq} with $\phi(x)=x$ 
		and MVOPs $P(x,n)$. 
		If $X(n)$ and $Y(n)$ are the coefficients of the $(n-1)$-term and $(n-2)$-term of $P(x,n)$ respectively, then
		\begin{equation}\label{fla A0n}
			n+X(n)A - A X(n+1)-B(n) = -(n+1+\nu)-\mathcal{H}(n)J\mathcal{H}^{-1}(n),
		\end{equation}
		\begin{equation}\label{fla Ad-1n}
			X(n)+[J,X(n)]=\mathcal{H}(n)(A^{\ast}-1) \mathcal{H}^{-1}(n-1),
		\end{equation}
		where $\mathcal{H}(n)$ and $B(n)$are as in \eqref{equation Hn} and \eqref{B C}. 
	\end{cor}
	\begin{proof}
		It follows from Proposition \ref{prop cal D y D dag} and Corollary \ref{coro v grado 1}, by taking into account the term $A_{0}(n)$ of $\mathcal{D}$.
		
		On the other hand, from the equality $P\cdot \mathcal{D}^{\dagger}= M^{\dagger}\cdot P$ with
		$$\mathcal{D}^{\dagger}= -\partial_x x - (1+\nu +J) \quad \text{and} \quad  M^{\dagger}=-(n+\nu+J+1)+A^{\dagger}_{-1}(n) \delta^{-1},$$
		we can obtained the equation \eqref{fla Ad-1n} from the above theorem. 
	\end{proof}
	
	\subsection*{Existence of the operator $D$}
	Families of matrix valued orthogonal polynomials which are eigenfunctions of a second order differential operator are of great importance, see e.g. \cite{Duran2009_2}, \cite{DuranG1}, \cite{GdIM}, \cite{KdlRR}.  Using the approach of \cite{KdlRR}, we get a symmetric second-order differential operator which preserves polynomials and its degree. For this we establish a conjugation with a diagonal matrix differential operator. 
	
	Let us consider matrix valued polynomials $F_2$, $F_1$, $F_0$ of degrees two, one and zero respectively and let us assume that we have a matrix valued second-order differential operator $D$ such that
	\begin{equation}
		\label{eq:form-differential-operator-general}
		Q\cdot D=\left(\tfrac{d^2Q}{dx^2}\right)(x) \, F_2(x) + \left(\tfrac{dQ}{dx}\right)(x) \, F_1(x) +Q(x) F_0(x).
	\end{equation}
	for a matrix valued polynomial $Q$. It follows from the definition of the matrix valued inner product \eqref{eq:HermitianForm} that a differential operator $D$ is symmetric with respect to $W$ if for all matrix valued polynomials $G,H$ we have
	$$\int_0^\infty (GD)(x)W(x)(H(x))^\ast\, dx = \int_0^\infty G(x)W(x)((HD)(x))^\ast dx.$$
	By \cite[Thm~3.1]{DuranG1}, this symmetry condition is equivalent to the following equations
	\begin{gather}
		\label{eq:symmetry-conditions}
		F_2(x)W(x) = W(x) \bigl( F_2(x)\bigr)^\ast, \qquad 
		2 \tfrac{d(F_2W)}{dx}(x) - F_1(x)W(x) = W(x) \bigl( F_1(x)\bigr)^\ast, \\
		\label{eq:symmetry-conditions2}
		\tfrac{d^2(F_2W)}{dx^2}(x) - \tfrac{d(F_1W)}{dx}(x) + F_0(x) W(x) = W(x) \bigl( F_0(x)\bigr)^\ast,
	\end{gather}
	and	the boundary conditions 
	\begin{gather}
		\label{eq:symmetry-boundary1}
		\lim_{x\to 0} F_2(x)W(x) = 0 = \lim_{x\to \infty} F_2(x)W(x), \\ 
		\label{eq:symmetry-boundary2}
		\lim_{x\to 0} F_1(x)W(x) - \tfrac{d(F_2W)}{dx}(x) = 0 = 
		\lim_{x\to \infty} F_1(x)W(x) - \tfrac{d(F_2W)}{dx}(x).
	\end{gather}
	
	We have the following lemma.
	\begin{lem}\label{ Lema DQ}
		The second order differential operator $$D_{Q}=\partial_{x}^2 x +\partial_x ( 1+\nu-x+J) -J$$ 
		is symmetric respect to the weight $T^{(\nu)}(x)$.
	\end{lem}
	\begin{proof}
		By taking into account that $F_2(x)=x$, $F_{1}(x)= 1+\nu -x+J$ and $F_{0}(x)=-J$, it is a straightforward computation see that 
		\eqref{eq:symmetry-boundary1}, \eqref{eq:symmetry-boundary2}, \eqref{eq:symmetry-conditions} and \eqref{eq:symmetry-conditions2} are satisfied for $D_Q$ and $T^{(\nu)}(x)$.
	\end{proof}
	
	\begin{prop} \label{Prop Gama}
		The second order differential operator 
		$$D=\partial_{x}^2 x +\partial_x ( (A-1)x +1+\nu+J) +A\nu + JA -J$$ 
		is symmetric respect to the weight $W(x)$. Moreover
		$$P\cdot D = \Gamma \cdot P \quad \text{where} \quad \Gamma(n)=A(n+\nu+1+J)-n-J.$$
	\end{prop}
	\begin{proof}
		It follows from Remark 4.1 in  \cite{koelink2019matrix}. 
	\end{proof}
	
	\subsection*{The Lie Algebra associated to $W$} \
	
	Recall that a Lie algebra $\mathfrak{g}$ is called reductive if its radical is equal to its center. 
	
	\begin{lem}\label{bracket D x }
		Let $A,J\in M_{N}(\mathbb{C})$ as in \eqref{A J} and 
		let us consider the operators
		{\small $$\mathcal{D} = \partial_x x + x(A-1), \, 
			\mathcal{D}^\dagger = -\partial_x x - (1+\nu+J), \, 
			D=\partial_{x}^2 x +\partial_x ( (A-1)x +1+\nu+J) +A\nu + JA -J.$$}
		If $x$ acts over matrix valued polynomials by right constant multiplication, 
		then we have that 
		$$[\mathcal{D},x] = -x,\qquad [\mathcal{D}^\dagger,x]=x,\qquad [\mathcal{D},\mathcal{D}^\dagger] = x, \qquad [D,x]=-\mathcal{D}+\mathcal{D}^{\dagger},$$
		$$[\mathcal{D},D]= -\mathcal{D}+ D - (1+\nu), \quad [\mathcal{D}^{\dagger},D]= \mathcal{D}^{\dagger}- D+ (1+\nu).$$
		The subjacent Lie algebra generated by $\{\mathcal{D}, \mathcal{D}^\dagger, D,x, I\}$	is isomorphic to the Lie algebra \break
		$\mathcal{A}=\langle x_1,x_{2},x_{3},x_{4},x_{5} \rangle$ with brackets
		$$[x_1,x_2]=-x_4,\quad [x_{1},x_4]=-x_4 ,\quad [x_2,x_4]=x_4, \quad [x_3,x_4]=-x_1+x_2,$$
		$$[x_1,x_3]= -x_2+ x_3 +x_4 - (1+\nu)x_5 \quad \text{and} \quad[x_2,x_3]= x_1 - x_3 -x_4 + (1+\nu)x_5,$$
		with correspondence:
		$$x_1\longmapsto \mathcal{D}+x,\quad x_2\longmapsto \mathcal{D}^{\dagger}+x, \quad x_3\longmapsto D,\quad x_4\longmapsto x, \quad x_5 \longmapsto 1.$$
	\end{lem}
	
	\begin{proof}	
		Let $\mathcal{D}'=\mathcal{D}+x$ and $\mathcal{D}'^{\dagger}=\mathcal{D}^{\dagger}+x$.
		It can be shown by direct computation that
		$$[\mathcal{D}',x]=-x, \quad [\mathcal{D}'^\dagger,x]= x, \quad [\mathcal{D}',\mathcal{D}'^{\dagger}]=-x, \quad [D,x]=-\mathcal{D}'+\mathcal{D}'^{\dagger},$$
		$$[\mathcal{D}',D]= -\mathcal{D}'^{\dagger}+ D +x - (1+\nu), \quad [\mathcal{D}'^{\dagger},D]= \mathcal{D}'- D -x + (1+\nu).$$
		Clearly, the subjacent Lie algebra associated to this representation, is the Lie algebra \break $\mathcal{A}=\langle x_1,x_{2},x_{3},x_{4},x_{5} \rangle$ with brackets
		$$[x_1,x_2]=-x_4,\quad [x_{1},x_4]=-x_4 ,\quad [x_2,x_4]=x_4, \quad [x_3,x_4]=-x_1+x_2,$$
		$$[x_1,x_3]= -x_2+ x_3 +x_4 - (1+\nu)x_5 \quad \text{and} \quad[x_2,x_3]= x_1 - x_3 -x_4 + (1+\nu)x_5,$$
		with correspondence 
		$$x_1\longmapsto \mathcal{D}+x,\quad x_2\longmapsto \mathcal{D}^{\dagger}+x, \quad x_3\longmapsto D,\quad x_4\longmapsto x, \quad x_5 \longmapsto 1,$$
		as desired.		
	\end{proof}

	We have the following structure result.
	
	\begin{prop}
		The Lie algebra $\mathcal{A}$ defined as above is a 5-dimensional reductive algebra with center of dimension two, given by $\mathcal{Z}_{\mathcal{A}}=\langle x_{1}+x_2- x_4,x_5\rangle$. Moreover, 
		$\mathcal{A}=[\mathcal{A},\mathcal{A}]\oplus \mathcal{Z}_{\mathcal{A}}$  with
		$[\mathcal{A},\mathcal{A}]$ isomorphic to $\mathrm{SL}(2,\mathbb{C})$. In particular, 
		$$\mathcal{C}_1= -4x_4(x_1-x_3-x_4+(1+\nu)x_5)+(x_4-x_1+x_2)^2, \quad \mathcal{C}_2= x_{1}+x_2- x_4\quad\text{and} \quad \mathcal{C}_3=x_5$$ 
		are Casimir elements of $\mathcal{A}$.
	\end{prop}
	\begin{proof}
		From Lie algebra theory (see \cite{knapp1996lie}), the radical of $\mathcal{A}$ can be computed from its Killing form, in this case we have that 
		$$\mathrm{Rad}(\mathcal{A})=\mathcal{Z}_{\mathcal{A}}=\langle x_{1}+x_2- x_4,x_5\rangle,$$
		and so, the Lie algebra $\mathcal{A}$ is reductive. 
		Now, from general theory, since $\mathcal{A}$ is reductive, we obtain
		$$\mathcal{A}=[\mathcal{A},\mathcal{A}] \oplus \mathcal{Z}_{\mathcal{A}}.$$
		In this case, we obtain that
		$$[\mathcal{A},\mathcal{A}]= \langle  x_4, \, x_1 -x_2,\,  x_1-x_3 -x_4 +(1+\nu) x_5\rangle.$$
		By taking $ a_1=x_4$, $a_2= x_1-x_2$ and $a_3=x_1-x_3-x_4+(1+\nu)x_5$ we obtain that 
		$$[a_1,a_2]=-2a_1, \quad [a_1,a_3]=a_1-a_2, \quad [a_2,a_3]=-a_3$$
		and so if we take $\hat{a}_2=a_1-a_2$ and $\hat{a}_3=-a_3$ we have 
		$$[a_1,\hat{a}_2]=2a_1, \quad [a_1,\hat{a}_3]=-\hat{a}_2, \quad [\hat{a}_2,\hat{a}_3]=2\hat{a}_3,$$
		and so $[\mathcal{A},\mathcal{A}] $ is isomorphic to $\mathrm{SL}(2,\mathbb{C})$ by consider the map $a_1\mapsto e_1$ $ \hat{a}_2 \mapsto e_2$ and $\hat{a}_3 \mapsto e_3$.
		In particular, the Casimir element of $\mathrm{SL}(2,\mathbb{C})$ given by $4e_1e_3+e_{2}^2$ induces a Casimir element of $[\mathcal{A},\mathcal{A}]$
		$$\mathcal{C}_{[\mathcal{A},\mathcal{A}]}= -4x_4(x_1-x_3-x_4+(1+\nu)x_5)+(x_4-x_1+x_2)^2.$$
		By taking into account that $\mathcal{C}_{[\mathcal{A},\mathcal{A}]}$ commutes with the central elements of $\mathcal{A}$, we obtain that $\mathcal{C}_{[\mathcal{A},\mathcal{A}]}$ commutes with all of the elements of  $\mathcal{A}$ and so is a Casimir element of $\mathcal{A}$. 
		
		The last assertion is clear, since the central elements always are Casimir elements of a given Lie algebra.
	\end{proof}

	\begin{rmk}
		Notice that 
		$$\mathcal{C}=\mathcal{C}_2+ (1+\nu )\mathcal{C}_3 = \mathcal{D}+\mathcal{D}^{\dagger}+x+(1+\nu) $$ 
		it is also a Casimir invariant of $\mathcal{A}$. 
		Hence, under the representation given by $\mathcal{D},\mathcal{D}^{\dagger}, D,x,1$, 
		the image of this Casimir satisfies
		$$\mathcal{C}= Ax-J.$$
		Thus, in terms of $M,M^{\dagger}$ and $L$, we obtain that
		\begin{equation}\label{Casimdisc.}
			\varphi^{-1}(\mathcal{C})=M+M^{\dagger}+L+(1+\nu).
		\end{equation}
		On the other hand, it can be check that
		$$\mathcal{C}_{[\mathcal{A},\mathcal{A}]}=\dfrac{1}{8}(A^{2}x^{2}+\nu^{2}+J^{2}+Ax-2\nu Ax-2xJA+2\nu J -1).$$
		Thus, $\mathcal{C}'=A^{2}x^{2}-2xJA+J+J^{2}$ it is also a Casimir since
		$$\mathcal{C}'=8\mathcal{C}_{[\mathcal{A},\mathcal{A}]}-(1-2\nu)\mathcal{C} - (\nu^2-1)\mathcal{C}_3,$$
		moreover, notice that the relation $[J,A]=A$ implies that
		$\mathcal{C}'=\mathcal{C}^2-\mathcal{C}$, and so we obtain that
		$$\mathcal{C}_{[\mathcal{A},\mathcal{A}]}=\tfrac{1}{8}\Big( \mathcal{C}^2-2\nu \mathcal{C} + (\nu^2-1)\mathcal{C}_3\Big).$$
		Therefore, in this representation the Casimir element corresponding to $\mathcal{C}_{[\mathcal{A},\mathcal{A}]}$ does not give more information than $\mathcal{C}$.
		Hence, in the rest of the paper we will only consider the Casimir element $\mathcal{C}$.
	\end{rmk}
	
	We can also consider the Lie subalgebra generated by $\{x_1,x_2,x_4,x_5\}$. 
	Notice that a representation of this kind of algebra was consider in the above section by taking $\phi(x)=x$.
	In this case, we have the following structure result.

	\begin{prop}
		The Lie subalgebra $\mathcal{A}' = \langle x_1,x_2,x_4,x_5 \rangle $ of $\mathcal{A}$ is isomorphic to $\mathfrak{g}_2\oplus \mathbb{C}^2 $ where $\mathfrak{g}_2$ is the 2-dimensional solvable Lie algebra with bracket $[e_1,e_2]=e_2$.
		In particular, $\mathcal{A}'$ has no non-central Casimir invariants. 
	\end{prop}
	
	\begin{proof}
		By taking the map 
		$$x_2\mapsto e_1, \quad x_4 \mapsto e_2, \quad x_1+x_2-x_4\mapsto e_3 \quad \text{and } \quad x_5 \mapsto e_4.$$
		we obtain an isomorphic algebra of $\mathcal{A}'$, in this case the only non-vanishing bracket of $\langle e_1,e_2,e_3,e_4\rangle$ is the bracket 
		$[e_1,e_2]=e_2$ and so $\mathcal{A}'$ is isomorphic to $\mathfrak{g}_2\oplus \mathbb{C}^2$, as asserted. 
		
		The last assertion is a consequence of $\mathfrak{g}_2$ has not central elements.
		Hence, the only Casimir invariants of $\mathcal{A}'$ are the central elements. 
	\end{proof}
	
	In the sequel, in order to simplify the notation, we will consider 
	\begin{equation}\label{BCH}
		B_{n}:=B(n),\quad C_{n}:=C(n), \quad \mathcal{H}_{n}:=\mathcal{H}(n), \quad \Gamma_n= \Gamma(n).    
	\end{equation}
	The following proposition is a consequence, of the relations between the brackets of $M,M^{\dagger},L,\Gamma$. 
	
	\begin{prop}\label{prop infinitas eq}
		Let $A,J\in M_{N}(\mathbb{C})$ be matrices as in \eqref{A J} and let $B_n,C_n,\mathcal{H}_n$ and $\Gamma_n$ be as in \eqref{BCH}. 
		Then,
		\begin{equation}\label{ML.1}
			B_n (A-1) - (A-1) B_{n+1} = 2 + \mathcal{H}_{n+1} J \mathcal{H}_{n+1}^{-1} - \mathcal{H}_n J \mathcal{H}_n^{-1},
		\end{equation}
		\begin{equation}\label{MdL0}
			B_n=[B_n,J] + \mathcal{H}_n (A^{t}-1) \mathcal{H}_{n-1}^{-1} - \mathcal{H}_{n+1} (A^{t}-1) \mathcal{H}_n^{-1},
		\end{equation}
		\begin{equation}\label{MdL-1}
			2 \mathcal{H}_n \mathcal{H}_{n-1}^{-1}= [\mathcal{H}_n \mathcal{H}_{n-1}^{-1},J] - B_n \mathcal{H}_n (A^{t}-1) \mathcal{H}_{n-1}^{-1} - \mathcal{H}_n (A^{t}-1) \mathcal{H}_{n-1}^{-1} B_{n-1},
		\end{equation}
		\begin{equation}\label{MMd0}
			B_n = -[J,\mathcal{H}_n J \mathcal{H}_n^{-1}] - \mathcal{H}_n (A^{t}-1) \mathcal{H}_{n-1}^{-1} (A-1) + (A-1) \mathcal{H}_{n+1} (A^{t}-1) \mathcal{H}_n^{-1},
		\end{equation}
		\begin{equation}\label{GamaL-1}
			[\Gamma, C_{n}\delta^{-1}] = \mathcal{H}_{n}(A-1)^*\mathcal{H}_{n-1}^{-1},
		\end{equation}
		\begin{equation}\label{GamaM0}
			[\Gamma_n,\mathcal{H}_{n}J \mathcal{H}_n^{-1}]= n+\Gamma_{n}+\mathcal{H}_{n}J \mathcal{H}_{n}^{-1},
		\end{equation}
		\begin{equation}\label{GamaMdag-1}
			[\Gamma,\mathcal{H}_{n}(A-1)^*\mathcal{H}_{n-1}^{-1} \delta^{-1}]=-\mathcal{H}_{n}(A-1)^*\mathcal{H}_{n-1}^{-1}. 
		\end{equation}
	\end{prop}
	\begin{proof}
		The equation \eqref{ML.1} is consequence of seeing the coefficient of $\delta^{1}$ in the bracket 
		relation $[M,L]=L$. 
		
		In the same way,  the equations \eqref{MdL0}, \eqref{MdL-1} are consequence of seeing the coefficients of $\delta^{0}$ and $\delta^{-1}$ in the bracket relation $[M^\dagger, L] = L$.
		
		On the other hand, the equation \eqref{MMd0} it follows from the bracket relation $[M,M^\dagger] = L$.
		The equation \eqref{GamaL-1} is a consequence of the coefficient of $\delta^{-1}$ in the bracket $[\Gamma,L] = -M+M^{\dagger}$.
		The equation \eqref{GamaM0} is obtained from the coefficient of $\delta^{0}$ in the bracket $[\Gamma,M] = M-\Gamma+(1+\nu)$.
		Finally, the equation \eqref{GamaMdag-1} can be obtained from the coefficient of $\delta^{-1}$ in the bracket relation
		$[\Gamma,M^{\dagger}] = -M^{\dagger}+\Gamma-(1+\nu)$.
	\end{proof}
	
	\medskip
	
	\begin{prop} 
		Let $A,J\in M_{N}(\mathbb{C})$ be matrices as in \eqref{A J}.
		Then,
		\begin{align*}
			[B_n,J] &=  B_n  - \overline{\Gamma_n} \overline{C_n} + \overline{C_n \Gamma_{n-1}} + \overline{\Gamma_{n+1} C_{n+1}} - \overline{C_{n+1} \Gamma_n}, \\
			[C_n,J]&= 2C_n+B_n (\overline{\Gamma_n} \overline{C_n} - \overline{C_n \Gamma_{n-1}}) + (\overline{\Gamma_n} \overline{C_n} - \overline{C_n \Gamma_{n-1}}) B_{n-1},
		\end{align*}
		where $B_n,C_n,\mathcal{H}_n$ and $\Gamma_n$ be as in \eqref{BCH}
	\end{prop}
	\begin{proof}
		By \eqref{MdL0} we have that 
		$$[B_n,J] = B_n - \mathcal{H}_n (A^{t}-1) \mathcal{H}_{n-1}^{-1} + \mathcal{H}_{n+1} (A^{t}-1) \mathcal{H}_n^{-1}.$$
		On the other hand, by \eqref{GamaL-1} and taking into account that $\mathcal{H}_{n}\in M_N(\mathbb{R})$, we obtain that
		\begin{align*}
			\mathcal{H}_{n}(A^t-1)\mathcal{H}_{n-1}^{-1} =  \overline{[\Gamma, C_{n}\delta^{-1}]}&= \overline{\Gamma_n} \overline{C_n} - \overline{C_n \Gamma_{n-1}},\\
			\mathcal{H}_{n+1}(A^t-1)\mathcal{H}_{n}^{-1} = \overline{[\Gamma, C_{n}\delta^{-1}]} &= \overline{\Gamma_{n+1} C_{n+1}} - \overline{C_{n+1} \Gamma_n},
		\end{align*}
		and so we have that
		$$	[B_n,J] =  B_n  - \overline{\Gamma_n} \overline{C_n} + \overline{C_n \Gamma_{n-1}} + \overline{\Gamma_{n+1} C_{n+1}} - \overline{C_{n+1} \Gamma_n},$$
		as asserted.
		
		For the second equality, notice that since $C_{n}=\mathcal{H}_{n}\mathcal{H}_{n-1}^{-1}$, 
		thus the equation \eqref{MdL-1}  implies that 
		$$[C_n,J]=2C_n+B_n \mathcal{H}_n (A^{t}-1) \mathcal{H}_{n-1}^{-1} + \mathcal{H}_n (A^{t}-1) \mathcal{H}_{n-1}^{-1} B_{n-1}.$$
		In the same way as above, we obtain that
		$$[C_n,J]=2C_n+B_n (\overline{\Gamma_n} \overline{C_n} - \overline{C_n \Gamma_{n-1}}) + (\overline{\Gamma_n} \overline{C_n} - \overline{C_n \Gamma_{n-1}}) B_{n-1},$$
		as desired.
	\end{proof}

	\section{Matrix entries of $P(x,n)$ as classical Laguerre polynomials}
	
	In this section we will give explicit expressions of the entries of the Laguerre-type MVOPs in terms of the classical scalar Laguerre polynomials. For this we use the approach of \cite{KdlRR}, \cite{ISMAIL2019235}, \cite{KR}, which consists in observing that a symmetric second order differential operator can be diagonalized via conjugation with an appropriate matrix valued function. The present situation is more involved than the previous cases because, although the differential operator can be diagonalized, the eigenvalue remains non-diagonal.
	
	In the rest of the section, the weight matrix $W(x)$ is as in the previous section.
	\subsection*{Step I: Diagonalizing the differential operator:} 
	Let $A,J\in M_{N}(\mathbb{C})$ be matrices as in \eqref{A J}. We can define the following auxiliary matrix valued polynomials
	\begin{equation}\label{def Q}
		Q(x,n):=P(x,n)L(x), \qquad \text{where $L(x)= e^{xA}$.}
	\end{equation}
	The polynomials $Q_n$ satisfy the following relations,
	$$ Q_n\cdot \mathcal{D}_Q = M \cdot Q_n,\qquad Q_n\cdot D_Q = \Lambda_n \cdot Q_n,\qquad Q_n\cdot \mathcal{C}_Q = M_{\mathcal C} \cdot Q_n,$$
	where 
	\begin{equation}\label{def DQs}
		\mathcal{D}_Q= L(x)^{-1} \mathcal{D} L(x), \qquad D_Q=L(x)^{-1} D L(x), \qquad \mathcal{C}_Q=L(x)^{-1} \mathcal{C} L(x).
	\end{equation}
	Here $\mathcal{D}$, $D$ are as in Proposition \ref{prop cal D y D dag}, Proposition \ref{Prop Gama} respectively and $\mathcal{C}=Ax-J$ . 
	In the following lemma, we show that $\mathcal{D}_Q$, $D_Q$ and $\mathcal{C}_Q$ are simple diagonal operators.
	\begin{lem}\label{lemma CQs}		
		The operators $\mathcal{D}_Q, D_Q$ and $\mathcal{C}_Q$ as in \eqref{def DQs}
		are given explicitly as follows: 
		$$\mathcal{D}_Q=  \partial_x x -x, \qquad 
		D_Q=\partial_{x}^2 x +\partial_x ( 1+\nu-x+J) -J, \qquad \mathcal{C}_Q=-J.$$
	\end{lem}
	
	\begin{proof}
		All of the equalities are follow directly from definition and the equation \eqref{corchete A J}.
	\end{proof}
	
	In the following proposition, we will use the expressions of $\mathcal{C}_Q$, $\mathcal{C}$ and $\mathcal{D}$ 
	in order to obtain an equation which relates $Q(x,n)$ and the recurrent matrix $H_n(A^{\ast}-1)H_{n-1}^{-1}$.
	
	\begin{prop}\label{prop Q's}
		The auxiliary functions $Q(x,n)$'s defined in \eqref{def Q}, 
		satisfy the following equation which depend on the squared norms:
		{\small \begin{equation*}
				-Q(x,n)J=xQ'(x,n)- (n+J)Q(x,n)+\mathcal{H}(n)(A-1)^\ast\mathcal{H}^{-1}(n-1)Q(x,n-1) \quad 	\text{for $n\ge 1$}
		\end{equation*}}
		and for $n=0$
		\begin{equation*}
			-Q(x,0)J=-J-xQ'(x,0).
		\end{equation*}
	\end{prop}
	\begin{proof}
		By Lemmas \ref{AJ lemma} and \ref{lemma CQs} we have that 	
		$$			-Q(x,n)J = Q(x,n)\cdot \mathcal{C}_Q = (M_{\mathcal{C}}\cdot P(x,n))e^{xA}$$
		where $M_{\mathcal{C}}$ is the discrete operator
		{\small $$M_{\mathcal{C}}=A \delta + X(n)A-AX(n+1)-J + \Big(Y(n)A-AY(n+1) + [J,X(n)] + (AX(n+1)-X(n)A)X(n)\Big)\delta^{-1}.$$}		
		By equation \eqref{fla A-1n}, we have
		{\small $$ (n-1)X(n)+Y(n)(A-1)-(A-1)Y(n+1)-\Big( n+X(n)A - A X(n+1)-B(n)\Big)X(n) = 0,$$ }
		and so 
		{\small $$Y(n)A-AY(n+1)  + (AX(n+1)-X(n)A)X(n)= X(n)+Y(n)-Y(n+1)-B(n)X(n).$$}
		By taking into account the relation $P\cdot x =L\cdot P$ with $L=\delta+B(n)+C(n)\delta^{-1}$, 
		thus we have that
		$$Y(n)=Y(n+1)+B(n)X(n)+C(n).$$
		Hence, we obtain that
		{\small	$$(M_{\mathcal{C}}\cdot P(x,n))e^{xA}= AQ(x,n+1) + (X(n)A-AX(n+1)-J)Q(x,n)+ (C(n) + [J,X(n)] + X(n))Q(x,n-1).$$}
		On the other hand, by Corollary \ref{coro v grado 1} we have that
		{\small \begin{eqnarray*}
				(M_{\mathcal{D}}\cdot P(x,n))e^{xA} &=& (A-1) Q(x,n+1)+ \Big(n+X(n)A - A X(n+1)-B(n)\Big)Q(x,n), \\
				(L \cdot P(x,n))e^{xA} &=& Q(x,n+1) + B(n)Q(x,n)+ C(n) Q(x,n-1),
		\end{eqnarray*}}
		where $B(n)$ and $C(n)$ are as in equation \eqref{B C}. Hence, we have that
		$$((M_{\mathcal{D}}+L)\cdot P(x,n))e^{xA}= A Q(x,n+1)+(n+X(n)A - A X(n+1))Q(x,n)+C(n)Q(x,n-1)$$
		and thus
		{\small	$$ (M_{\mathcal{C}}\cdot P(x,n))e^{xA}= ((M_{\mathcal{D}}+L)\cdot P(x,n))e^{xA}-(n+J)Q(x,n)+ (X(n)+[J,X(n)])Q(x,n-1).$$}
		Notice that if we denote $L(x)=e^{xA}$, the Lemma \ref{lemma CQs} implies that
		{\small$$((M_{\mathcal{D}}+L)\cdot P(x,n))L(x)=(P(x,n)\cdot(\mathcal{D}+x))L(x)= Q(x,n)\cdot L^{-1}(x)(\mathcal{D}+x))L(x)= Q(x,n)\cdot \partial_{x}x,$$}
		and by \eqref{fla Ad-1n}, $X(n)+[J,X(n)]=\mathcal{H}(n)(A-1)^\ast \mathcal{H}^{-1}(n-1)$. 
		Thus, we obtain that
		{\small $$-Q(x,n)J= xQ'(x,n)- (n+J)Q(x,n)+\mathcal{H}(n)(A-1)^\ast \mathcal{H}^{-1}(n-1),$$}
		as asserted.
		Finally, the equation for $n=0$, it follows from equation \eqref{rel L J}.
	\end{proof}

	\subsection*{Step II: Diagonalizing the eigenvalue $\Lambda_n$:} 
	Although the differential operator $D_Q$ is a diagonal operator, the system of equations given by $Q_n\cdot D_Q = \Lambda_n Q_n$ is not a decoupled system since $\Lambda_n$ is a lower triangular matrix. However, $\Lambda_n$ can be diagonalized in a somewhat simple way:
	\begin{equation}
		\label{Propiedad Kn}
		\Gamma_n=A(n+\nu+J+1) - (n+J) = K_n \Lambda_n K_n^{-1},
	\end{equation}
	where $\Lambda_n $ is the diagonal matrix $\Lambda_n = -(n+J)$.   
	\begin{lem}\label{lem Kn}
		We can choose $K_n$ such that is a lower triangular matrix with diagonal elements $(K_n)_{i,i}=1$ for all $i=1,\ldots,N$. 
		Moreover, in this case $K_n$ is the following matrix: 
		\begin{equation}\label{def Kn}
			(K_n)_{i,j}=
			\begin{cases}
				(-1)^{i-j}\tfrac{(n+\nu + j+1)_{i-1}}{(i-j)!} \prod_{k=j}^{i-1}a_k  & i> j, \\
				1 & i=j, \\
				0 & i<j.
			\end{cases}
		\end{equation}
	\end{lem}
	\begin{proof}
		Since $A(n+\nu+J+1) - (n+J)$ is lower triangular, the its characteristic polynomial satisfy 
		$$\det\Big(\lambda I -\Big(A(n+\nu+J+1) - (n+J)\Big)\Big)=\det(\lambda I +nI+J)=\prod_{r=1}^{N}(\lambda+n+r).$$
		Then, the eingevalues of $A(n+\nu+J+1) - (n+J)$ are 
		$$\lambda_1=-(n+1), \ldots, \lambda_N=-(n+N) \quad \text{with multiplicity $1$}.$$
		Since $K_n$ diagonalizes $A(n+\nu+J+1) - (n+J)$, 
		then its $r$-th column can be obtained from the eigenspace correspond 
		to $\lambda_r=-(n+r)$, that is
		$$Nu(\lambda_r I-A(n+\nu+J+1) + n+J)=Nu(-A(n+\nu+J+1) +J-rI), \quad \text{ for $r=1, \ldots,N$}.$$
		We can obtain \eqref{def Kn} by a straightforward computation of these eigenspaces. 
	\end{proof}
	With the matrix $K_n$ as in \eqref{def Kn}, we will consider the following matrix polynomial 
	\begin{equation}\label{def R}
		R(x,n):=K_n^{-1}Q(x,n).
	\end{equation}
	The following result that shows a relationship between the non-zero matrix entries of $R(x,n)$ and generalized Laguerre polynomials.
	
	\begin{thm}\label{Coef R Laguerre}
		Let $n\in \mathbb{N}$ and let $\nu>0$.
		The matrix elements of $R(x,n)$ are multiples of scalar Laguerre functions 
		\begin{equation}\label{lagpol}
			R(x,n)_{i,j}= 
			\begin{cases}
				L^{(\nu + j)}_{n+i-j}(x) \xi(n,i,j) & \text{if $n+i-j \geq 0$} \\
				0 & \text{if $n+i-j < 0$}.
			\end{cases}
		\end{equation}
	\end{thm}
	\begin{proof}
		Notice that the polynomials  $R(x,n)$'s are eigenfunctions of  $D_Q=\partial_{x}^2 x + \partial_x ( 1+\nu-x+J) -J$ 
		with associated eigenvalues $\Lambda_n=-(n+J)$ where $J$ is the diagonal matrix $\mathrm{diag}(1,2,\ldots,N)$. 
		If we look the $(i,j)$-entry of
		\begin{equation*}
			xR''(x,n) + R'(x,n) (1+\nu-x+J) + (nR(x,n) + JR(x,n) -R(x,n)J)=0,
		\end{equation*}
		we obtain that the following expression
		$$xR''(x,n)_{i,j} + \sum_{k=1}^{N} R'(x,n)_{i,k} (1+\nu-x+J)_{k,j} + nR(x,n)_{i,j} + \sum_{k=1}^{N} (J_{i,k}R(x,n)_{k,j} - R(x,n)_{i,k}J_{k,j}) = 0.$$
		Since $J=\mathrm{diag}(1,2,\ldots,N)$, the above equality is equivalent to
		$$xR''(x,n)_{i,j} +  R'(x,n)_{i,j} (1+\nu-x+j) + (n+i-j)R(x,n)_{i,j} = 0.$$
		Hence, we obtain that 
		$$R(x,n)_{i,j}= L^{(\nu+j)}_{n+i-j}(x) \xi(n,i,j),$$
		as asserted.
		
		On the other hand,  it is well-known that if $M<0$ the only solution for the differential equation 
		$$x P''(x) +  P'(x) (1+\nu-x) + M P(x) = 0,$$
		is $P(x)=0$ and since
		$$xR''(x,n)_{i,j} +  R'(x,n)_{i,j} (1+\nu-x+j) + (n+i-j)R(x,n)_{i,j} = 0,$$
		we obtain that $R(x,n)_{i,j}=0$ if $n+i-j<0$, as asserted.
	\end{proof}
	
	\begin{rmk}
		Notice that by the above theorem the are only defined for $n+i-j \geq 0$, so we extend its definition as follows
		\begin{equation}\label{def extend xi}
			\xi(n,i,j)=0 \qquad \text{if $n+i-j < 0$}.
		\end{equation}
	\end{rmk}
	
	It is well-known that the generalized Laguerre polynomial satisifes
	$$L^{(\alpha)}_{N}(0)=\frac{(\alpha+1)_N}{N!},$$
	where $(a)_N$ is the Pochhammer symbol defined by
	\begin{equation}\label{Pochhamer symbol}
		(a)_{N}= 
		\begin{cases}
			1 & \text{if $N=0$},\\
			a(a+1)\cdots(a+N-1) & \text{if $N>0$}.
		\end{cases}	
	\end{equation}
	As a direct consequence of the above theorem we obtain the following corollary.
	\begin{cor}\label{coro R(0,n)}
		Let $n\in \mathbb{N}$ and let $\nu>0$.
		Then, the coefficients of $R(0,n)$ satisfy
		\begin{equation}\label{lagpol x=0}
			R(0,n)_{i,j}= 
			\begin{cases}
				\frac{(\nu+j+1)_{n+i-j}}{(n+i-j)!} \xi(n,i,j) & \text{if $n+i-j \geq 0$,} \\
				0 & \text{if $n+i-j < 0$}.
			\end{cases}
		\end{equation}
		In particular, the $(i,n+i)$-th and $(i,n+i-1)$-th coordinates of $R(0,n)$ satisfy $$R(0,n)_{i,n+i}=\xi(n,i,n+i) \quad \text{and}\quad R(0,n)_{i,n+i-1}=(n+\nu+i)\xi(n,i,n+i-1).$$
	\end{cor}
	Now we need to identify the coefficients $\xi(n,i,j)$. 
	For this, we exploit the relation in Proposition \ref{prop Q's}. In the following lemma, we observe that the factor $\mathcal{H}(n)(A-1)^\ast\mathcal{H}^{-1}(n-1)$ in this relation is turned into a diagonal matrix via multiplication by appropriate matrices $K_n$. This will allow us to obtain a simple recursion for $\xi(n,i,j)$. For this purpose, we define the following 
	\begin{equation}\label{GI}
		G(n):= K_n^{-1} \mathcal{H}_{n} (A^{\ast}-1)\mathcal{H}_{n-1}^{-1}K_{n-1}\qquad I(n):= K_n^{-1} \mathcal{H}_n J \mathcal{H}_n^{-1}K_{n},
	\end{equation}	
	where $\mathcal{H}_n$ and $K_n$ are as in \eqref{BCH} and \eqref{def Kn}, respectively. 
	\begin{lem}\label{diag GI}
		For $n\geq 1$ and let  $G(n)$, $I(n)$ be the matrices as in \eqref{GI}.
		Then, $G(n)$ is diagonal and $I(n)$ satisfies
		$$(I(n))_{i,i} = i, \qquad (I(n))_{i,j}=0 \ \text{ for } i\neq j \text{ and } i \neq j-1.$$
		Moreover
		\begin{equation}\label{eq GI diag}
			(G(n))_{i,i}= (\mathcal{H}_{n} (A^{\ast}-1)\mathcal{H}_{n-1}^{-1})_{i,i} \quad \text{and} \quad (I(n))_{i,i+1}=(\mathcal{H}_n J \mathcal{H}_n^{-1})_{i,i+1}.
		\end{equation}
	\end{lem}
	
	\begin{proof}
		By equation \eqref{GamaMdag-1}, $$ [\Gamma_n,\mathcal{H}_n(A^{\ast}-1)\mathcal{H}_{n-1}^{-1}\delta^{-1}]=-\mathcal{H}_n(A^{\ast}-1)\mathcal{H}_{n-1}^{-1}.$$
		Then, 
		$$K_n^{-1} \Gamma_n(\mathcal{H}_n(A^{\ast}-1)\mathcal{H}_{n-1}^{-1})K_{n-1}-K_n^{-1}(\mathcal{H}_n(A^{\ast}-1)\mathcal{H}_{n-1}^{-1})\Gamma_{n-1}K_{n-1}=-G(n).$$
		Then, by $K_n^{-1}\Gamma_nK_n=\Lambda_n=-(n+J)$ we obtain that
		$$\Lambda_{n} G(n)-G(n) \Lambda_{n-1}=-G(n).$$
		The assertion it follows by observing the $(i,j)$-entry for $i\neq j$.
		On the other hand, from definition we have that
		$$K_n G(n) K_{n-1}^{-1}=  \mathcal{H}_{n} (A^{\ast}-1)\mathcal{H}_{n-1}^{-1}.$$
		Since $K_n$ and $K_{n-1}^{-1}$ are both lower triangular matrices with $1$'s in its diagonal, then the left term in the above equation is a lower triangular matrix with $(i,i)$-th coordinate equal to $(G(n))_{i,i}$ and so
		$$(\mathcal{H}_{n} (A^{\ast}-1)\mathcal{H}_{n-1}^{-1})_{i,i}=(K_n G(n) K_{n-1}^{-1})_{i,i}=(G(n))_{i,i},$$
		as asserted.
		
		Now, for the second assertion, we can proceed as in the same way, 
		in this case by equation \eqref{GamaM0}, we have that 
		$$[\Gamma_n,\mathcal{H}_nJ \mathcal{H}_n^{-1}]=n+\Gamma_n + \mathcal{H}_nJ\mathcal{H}_n^{-1}.$$
		Then, we obtain that
		$$K_n^{-1}\Gamma_n \mathcal{H}_nJ \mathcal{H}_n^{-1}K_n-K_n^{-1} \mathcal{H}_nJ \mathcal{H}_n^{-1}\Gamma_n K_n = K_n^{-1}(n+\Gamma_n+\mathcal{H}_nJ \mathcal{H}_n^{-1})K_n.$$
		Hence, by $K_n^{-1}\Gamma_nK_n=\Lambda_n=-(n+J)$ we obtain that
		$$\Lambda_n I(n)-I(n) \Lambda_n = n+ \Lambda_n + I(n)).$$
		The proof follows by observing the $(i,j)$-entry in the above matrix equality.
		Finally, from definition we have 
		$$K_nI(n)K_{n}^{-1}=\mathcal{H}_n J \mathcal{H}_n^{-1}.$$
		So, in general we have that 
		{\small $$(K_nI(n)K_{n}^{-1})_{i,j}= \sum_{k=1}^{N} (K_n)_{i,k} (I(n)K_{n}^{-1})_{k,j} \quad \text{and}\quad (I(n)K_{n}^{-1})_{k,j}=(I(n))_{k,k}(K_{n}^{-1})_{k,j}+(I(n))_{k,k+1}(K_{n}^{-1})_{k+1,j}.$$}
		By taking into account that $(K_n)_{i,j}=(K_n^{-1})_{i,j}=0$ for $j>i$, 
		if $j=i+1$ we obtain that
		{\small $$(K_nI(n)K_{n}^{-1})_{i,i+1}=\sum_{k=1}^{N} (K_n)_{i,k} (I(n)K_{n}^{-1})_{k,j}=\sum_{k=1}^{i} (K_n)_{i,k} (I(n)K_{n}^{-1})_{k,i+1}= (I(n))_{i,i+1}(K_{n}^{-1})_{i+1,i+1},$$}
		by taking into account that $(K_{n}^{-1})_{i+1,i+1}=1$ for any $i$. Therefore, we have that \break
		$(K_nI(n)K_{n}^{-1})_{i,i+1}=(I(n))_{i,i+1}$, as desired.
	\end{proof}

	\medskip
	
	The following proposition is a consequence of the relation given by the Casimir operator and Proposition 6.4
	
	\begin{prop}\label{recurrencia xi general}
		Let $\xi(n,i,j)$ as in Theorem \ref{Coef R Laguerre}.  
		If $n+i-j> 0$, then the constants $\xi(n,i,j)$'s satisfy the following: 
		\begin{enumerate}[($a$)]
			\item $\xi(0,i,j)={(K_0^{-1})}_{i,j} \tfrac{(i-j)!}{(\nu+j+1)_{i-j}}$,
			\item If $i=1$ and $n>0$ then 
			\begin{equation}\label{recurrence xi i=1}
				\xi(n,1,j) =\tfrac{(G(n))_{1,1} }{ n+\nu+1} \xi(n-1,1,j),
			\end{equation}
			\item If $i>1$ and $n>0$, then 
			\begin{equation}\label{recurrence xi}
				\xi(n,i,j) = a_{i-1}\xi(n,i-1,j) +\tfrac{(G(n))_{i,i} }{ n+\nu+i} \xi(n-1,i,j),
			\end{equation}
		\end{enumerate}	
		with $G(n)$ as in \eqref{GI}.
	\end{prop}
	\begin{proof}
		For the first assertion, recall that 
		{\small $$R(0,0)=K_0^{-1}Q(0,0)=K_0^{-1}P(0,0)=K_0^{-1},$$} 
		thus we have that  
		{\small $${(K_0^{-1})}_{i,j}=\xi(0,i,j)L^{(\nu+j)}_{i-j}(0).$$}
		By taking into account that $L^{(\nu)}_N(0)=\tfrac{(\nu+1)_N}{N!}$ for any $N$, hence we obtain 
		{\small $${(K_0^{-1})}_{i,j} \tfrac{(i-j)!}{(\nu+j+1)_{i-j}}=\xi(0,i,j),$$}
		as asserted.
		
		For items ($b$) and ($c$), recall that Proposition \ref{prop Q's} implies that 
		{\small $$-Q(0,n)J = -(n+J)Q(0,n) + \mathcal{H}(n) (A^{\ast}-1) \mathcal{H}(n-1)^{-1} Q(0,n-1).$$}
		Then,
		{\small $$-K_n R(0,n)J = -(n+J)K_n R(0,n) + \mathcal{H}(n) (A^{\ast}-1) \mathcal{H}(n-1)^{-1} K_{n-1}R(0,n-1),$$}
		and so we obtain that
		{\small $$-R(0,n)J = -K_{n}^{-1}(n+J)K_n R(0,n) +G(n) R(0,n-1)$$}
		where $G(n)=K_{n}^{-1}\mathcal{H}(n) (A^{\ast}-1) \mathcal{H}(n-1)^{-1} K_{n-1}$.
		By recalling that 
		$$K_{n}^{-1}(n+J)K_n R(0,n)=(n+J)-A(n+\nu+J+1),$$
		we obtain that 
		{\small $$-R(0,n)J = -(n+J) R(0,n)+ A(n+\nu+J+1) R(0,n)+ G(n)R(0,n-1).$$}
		Thus, in terms of coordinates we have that
		{\small 
			\begin{equation}\label{eq j R}
				-j R(0,n)_{i,j} = -(n+i)R(0,n)_{i,j} + \sum_{k=1}^{N}(A(n+\nu + J+1))_{i,k}R(0,n)_{k,j} +  \sum_{k=1}^{N} (G(n))_{i,k}R(0,n-1)_{k,j}.
			\end{equation}
		}
		Since $A$ has only non-zero entries $a_i$'s in the place $(i,i+1)$ and $n+\nu + J+1$ is diagonal, we obtain that
		\begin{equation}\label{entradas AJ}
			\sum_{k=1}^{N}(A(n+\nu + J+1))_{i,k}R(0,n)_{k,j}=
			\begin{cases}
				(a_{i-1}(n+\nu + i))R(0,n)_{i-1,j} & \text{if $i>1$}, \\
				0 & \text{if $i=1$}.
			\end{cases}
		\end{equation}	
		Thus, for $i>1$ the equation \eqref{eq j R} takes the form	
		{ \small $$-j R(0,n)_{i,j}=  -(n+i)R(0,n)_{i,j} + a_{i-1}(n+\nu + i)R(0,n)_{i-1,j}+  (G(n))_{i,i}R(0,n-1)_{i,j}$$} 
		where in the last term of the equality, we use Lemma \ref{diag GI}. 
		By Theorem \ref{Coef R Laguerre}, we have that
		{\small $$
			(n+i-j) L^{(\nu+j)}_{n+i-j}(0) \xi(n,i,j) = L^{(\nu+j)}_{n+i-1-j}(0) \Big(a_{i-1}(n+\nu+i)   \xi(n,i-1,j)+ (G(n))_{i,i} \xi(n-1,i,j)\Big).$$}
		Therefore, since $n+i-j>0$ we obtain a similar expression of \eqref{recurrence xi} by multiplication for $((n+i-j) L^{(\nu+j)}_{n+i-j}(0))^{-1}$, that is
		$$			\xi(n,i,j) = C_{1}(n,i,j)\xi(n,i-1,j) +C_{2}(n,i,j) \xi(n-1,i,j),$$
		where
		$$C_1(n,i,j) = \tfrac{(a_{i-1}(n+\nu+i))  L^{(\nu+j)}_{n+i-1-j}(0)}{(n+i-j) L^{(\nu+j)}_{n+i-j}(0)} \quad \text{and} \quad C_2(n,i,j) = \tfrac{(G(n))_{i,i} L^{(\nu+j)}_{n-1+i-j}(0)}{(n+i-j) L^{(\nu+j)}_{n+i-j}(0)}.$$
		To finish the proof, recall that
		$$L^{(\alpha)}_{N}(0)=\frac{(\alpha+1)_N}{N!},$$
		where $(a)_N$ is the Pochhammer symbol defined by
		$$(a)_{N}= 
		\begin{cases}
			1 & \text{if $N=0$},\\
			a(a+1)\cdots(a+N-1) & \text{if $N>0$}.
		\end{cases}	$$
		Hence, we have that 
		$$\frac{L^{(\nu+j)}_{n+i-1-j}(0)}{L^{(\nu+j)}_{n+i-j}(0)}= \frac{n+i-j}{n+\nu+i}.$$
		Therefore we obtain \eqref{recurrence xi}, as asserted.
		
		Now, for  $i=1$ the equation \eqref{eq j R} takes the form
		$$					-j R(0,n)_{1,j} = -(n+1)R(0,n)_{i,j} + (G(n))_{1,1}R(0,n-1)_{1,j}.$$
		So, the equation \eqref{recurrence xi i=1} it can be prooved in a similar way as in the case $i>1$. 
	\end{proof}
	\begin{rmk}\label{rem item a}
		Notice that the item ($a$) in the above proposition still holds for $n+i-j=0$ since in this case we just use the definition of $\xi$'s.
	\end{rmk}
	
	Now, we are going to study the case $n+i-j=0$. 
	In this case, we have the following result.
	
	\begin{prop}\label{formula j=n+i}
		Let $n\in \mathbb{N}_0$ and let $I(n), \,G(n)$ be as in \eqref{GI}. 
		If $i\ge1$, then the constants $\xi$'s satisfy the following:
		\begin{enumerate}[($a$)]
			\item $\xi(0,i,i)= 1$ and $\xi(1,i,i+1)= I(0)_{i,i+1}$.
			\item If $i=1$ and $n>0$, then 
			{\small \begin{equation*}
					N_1(n,i) \xi(n,1,n+1)= N_2(n,i) \xi(n-1,2,n+1), \quad \text{ with }
			\end{equation*} }
			{\small $$N_1(n,i)=(\nu+2n+3) \tfrac{(G(n+1))_{1,1} }{ n+\nu+2} + (n+2+\nu) + {I(n)}_{1,2} (\nu+2n+2) a_1, \quad N_2(n,i)={I(n)}_{1,2} (\nu+2n+2)\tfrac{(G(n))_{2,2} }{ n+\nu+2}.$$}
			\item If $i>1$ and $n>0$, then 
			\begin{equation}\label{rec j=n+i}
				M_{1}(n,i)\xi(n+1,i-1,n+i)= M_2(n,i)\xi(n,i,n+i)+M_3(n,i)\xi(n-1,i+1,n+i),
			\end{equation}
			{\small $$ \quad \text{ with } \quad M_{1}(n,i)=a_{i-1}\Big( i(n+1+\nu+i)_{i-2}-(n+\nu+i)\Big),$$} 
			{\small$$M_{2}(n,i)=(G(n+1))_{i,i} + (n+\nu+i+1), \quad  M_3(n,i)=(I(n))_{i,i+1} G(n)_{i+1,i+1}.$$}
		\end{enumerate}
	\end{prop}
	
	\begin{proof}
		The first assertion of item ($a$) it follows from item ($a$) of \ref{recurrencia xi general} (see Remark \ref{rem item a}). Indeed, since $K_n^{-1}$ has a diagonal of $1$'s, in this case we have that
		$$\xi(0,i,i)= (K_0^{-1})_{i,i}\tfrac{0!}{(\nu+i+1)_0}=(K_0^{-1})_{i,i}=1.$$ 
		
		By Proposition \ref{prop cal D y D dag}, we have
		$\mathcal{D} = \partial_x x + x(A-1) $ and $M = (A-1)\delta-(n+1+\nu)-\mathcal{H}(n)J\mathcal{H}^{-1}(n)$ satisfies 
		$$P_n \cdot \mathcal{D} = M \cdot P_n.$$
		By recalling that $(P_n \cdot \mathcal{D})(0)=0$, 
		we obtain that 
		{\small $$(A-1)P(0,n+1)-(n+1+\nu)P(0,n)-\mathcal{H}_nJ\mathcal{H}^{-1}_n P(0,n)=0,$$ }
		since $P(0,m)=K_m R(0,m)$, we have that
		\begin{equation}\label{eq matrix IR}
			K_n^{-1}(A-1)K_{n+1} R(0,n+1)-(n+1+\nu)R(0,n)-I(n) R(0,n)=0
		\end{equation}
		where $I(n)=K_n^{-1}\mathcal{H}_n J\mathcal{H}_n ^{-1}K_n$. 
		
		If we consider the expression \eqref{eq matrix IR} with $n=0$, we arrive to
		\begin{equation*}
			0 = K_0^{-1}(A-1)K_{1} R(0,1)-K_0^{-1}(1+\nu)K_0 R(0,0)-I(0) R(0,0).
		\end{equation*}
		By taking the $(i,i+1)$-th coordinate, we obtain that
		\begin{equation*}
			0 = - R(0,1)_{i,i+1}-(1+\nu)R(0,0)_{i,i+1}-I(0)_{i,i} R(0,0)_{i,i+1}- I(0)_{i,i+1}R(0,0)_{i+1,i+1}.
		\end{equation*}
		From Theorem \ref{Coef R Laguerre} we have 
		\begin{equation*}
			\xi(1,i,i+1)= I(0)_{i,i+1}\xi(0,i+1,i+1)=I(0)_{i,i+1}.
		\end{equation*}
		as asserted.
		
		\medskip 
		
		Now, let us consider $n>0$ in the equation \eqref{eq matrix IR}. 
		By taking into account the $(i,n+i)$-coordinate in \eqref{eq matrix IR},
		
		{\small $$	\sum_{r=1}^{N}{(K_n^{-1}(A-1)K_{n+1})}_{i,r} {R(0,n+1)}_{r,n+i}-(n+1+\nu) {R(0,n)}_{i,n+i}-\sum_{r=1}^{N}{I(n)}_{i,r} {R(0,n)}_{r,n+i}=0.$$ }
		By Lemma \eqref{diag GI} and since ${R(0,n+1)}_{r,n+i}= 0$ for $r<i-1$ we have that 
		{\small \begin{equation*}\label{formula diag}
				\sum_{r=i-1}^{N}{(K_n^{-1}(A-1)K_{n+1})}_{i,r} {R(0,n+1)}_{r,n+i}-(n+1+\nu+i) {R(0,n)}_{i,n+i}-{I(n)}_{i,i+1} {R(0,n)}_{i+1,n+i}=0.
		\end{equation*} }
		
		On the other hand, since the matrix $K_n^{-1}(A-1)K_{n+1}$ is lower triangular we obtain that
		{\small $$	\sum_{r=i-1}^{N}(K_n^{-1}(A-1)K_{n+1})_{i,r} {R(0,n+1)}_{r,n+i}= \sum_{r=i-1}^{i}(K_n^{-1}(A-1)K_{n+1})_{i,r} {R(0,n+1)}_{r,n+i}.$$ }
		By taking into account that
		$(K_n^{-1}(A-1)K_{n+1})_{i,i-1}= a_{i-1} -(K_n^{-1}K_{n+1})_{i,i-1}$ and  \break 
		$ (K_n^{-1}(A-1)K_{n+1})_{i,i}=-1$,
		we obtain that
		\begin{eqnarray*}
			0 &=& \big(a_{i-1} -(K_n^{-1}K_{n+1})_{i,i-1}\big) R(0,n+1)_{i-1,n+i}- R(0,n+1)_{i,n+i}  \\
			&-&(n+1+\nu+i){R(0,n)}_{i,n+i}-{I(n)}_{i,i+1} {R(0,n)}_{i+1,n+i}. 
		\end{eqnarray*}
		By Corollary \ref{coro R(0,n)} we have
		\begin{eqnarray*}
			0 &=& \big(a_{i-1} -(K_n^{-1}K_{n+1})_{i,i-1}\big) \xi(n+1,i-1,n+i) - (n+1+\nu+i) \xi(n+1,i,n+i) \\
			&-&(n+1+\nu+i) \xi(n,i,n+i) -{I(n)}_{i,i+1} (n+\nu +i+1)\xi(n,i+1,n+i).
		\end{eqnarray*}
		
		On the other hand, by Proposition \ref{recurrencia xi general} if $n+i-j>0$ we know that 
		$$\xi(n,i,j) = a_{i-1}\xi(n,i-1,j) +\tfrac{(G(n))_{i,i} }{ n+\nu+i} \xi(n-1,i,j).$$
		Thus, in particular we obtain 
		$$\xi(n+1,i,n+i) = a_{i-1}\xi(n+1,i-1,n+i) +\tfrac{(G(n+1))_{i,i} }{ n+1+\nu+i} \xi(n,i,n+i),$$
		$$\xi(n,i+1,n+i) = a_{i} \xi(n,i,n+i) + \tfrac{(G(n))_{i+1,i+1})}{n+\nu+i+1}) \xi(n-1,i+1,n+i)$$
		and so we have 
		{\small \begin{eqnarray*}
				& &	\big(a_{i-1} -{(K_n^{-1}K_{n+1})}_{i,i-1}\big) \xi(n+1,i-1,n+i)\\
				&&- (\nu+n+i+1) \Big(a_{i-1}\xi(n+1,i-1,n+i) +\tfrac{(G(n+1))_{i,i} }{ n+1+\nu+i} \xi(n,i,n+i)\Big)\\
				&&= (n+1+\nu+i) \xi(n,i,n+i) \\
				&&+ {I(n)}_{i,i+1} (n+1+\nu+i)\Big( a_{i} \xi(n,i,n+i) + \tfrac{(G(n))_{i+1,i+1}}{n+\nu+i+1} \xi(n-1,i+1,n+i)\Big).
		\end{eqnarray*}}
		Hence 
		
		\begin{equation*}
			M_{1}(n,i)\xi(n+1,i-1,n+i)= M_2(n,i)\xi(n,i,n+i)+M_3(n,i)\xi(n-1,i+1,n+i),
		\end{equation*}
		with $M_{1}(n,i)=-{(K_n^{-1}K_{n+1})}_{i,i-1} - a_{i-1}(\nu+n+i)$,  $M_{2}(n,i)=(G(n+1))_{i,i} + (n+\nu+i+1),$ and $M_3(n,i)=(I(n))_{i,i+1} G(n)_{i+1,i+1}.$
		To finish the proof, notice that since $K_n^{-1}$ and $K_{n+1}$ are both lower triangular matrix with $1$'s in the diagonal, we have that
		$$(K_{n}^{-1}K_{n+1})_{i,i-1}=(K_{n}^{-1})_{i,i-1}+(K_{n+1})_{i,i-1}\quad \text{and} \quad (K_{n}^{-1})_{i,i-1}=-(K_{n})_{i,i-1},$$
		by definition of $K_n$'s we have that
		$$(K_{n}^{-1}K_{n+1})_{i,i-1}=a_{i-1}\Big( (n+\nu+i)_{i-1}-(n+1+\nu+i)_{i-1}\Big)=-a_{i-1} i(n+1+\nu+i)_{i-2}.$$
		Hence, $$M_{1}(n,i)=a_{i-1}\Big( i(n+1+\nu+i)_{i-2}-(n+\nu+i)\Big),$$
		as asserted.
		
		For the case $i=1$, $j=n+1$ in the expression \eqref{eq matrix IR} we obtain
		{\small \begin{equation*}
				0 = - {R(0,n+1)}_{1,n+1}-(n+2+\nu) {R(0,n)}_{1,n+1}-{I(n)}_{1,2} {R(0,n)}_{2,n+1},
		\end{equation*} }
		then, from Corollary \ref{coro R(0,n)} we have
		{\small \begin{equation}\label{eq reemp}
				0 = - (\nu+2n+3)\xi(n+1,1,n+1)-(n+2+\nu) \xi(n,1,n+1)-{I(n)}_{1,2} (\nu+2n+2) \xi(n,2,n+1).
		\end{equation} }
		As a consequence of Proposition \ref{recurrencia xi general}	we have that	 
		$$\xi(n,i,j) = a_{i-1}\xi(n,i-1,j) +\tfrac{(G(n))_{i,i} }{ n+\nu+i} \xi(n-1,i,j) \quad \text{for $j<n+i$},$$
		which implies that	
		$$\xi(n,2,n+1) = a_{1}\xi(n,1,n+1) +\tfrac{(G(n))_{2,2} }{ n+\nu+2} \xi(n-1,2,n+1),$$
		and 
		$$	\xi(n,1,j) =\tfrac{(G(n))_{1,1} }{ n+\nu+1} \xi(n-1,1,j) \quad \text{for $j<n+1$}.$$ 
		In particular    $$\xi(n+1,1,n+1)=\tfrac{(G(n+1))_{1,1} }{ n+\nu+2} \xi(n,1,n+1).$$ 
		Thus, we obtain     
		$$				N_1(n,i) \xi(n,1,n+1)= N_2(n,i) \xi(n-1,2,n+1),$$
		with 
		{\small $$N_{1}(n,i)=	(\nu+2n+3) \tfrac{(G(n+1))_{1,1} }{ n+\nu+2} + (n+2+\nu) + {I(n)}_{1,2} (\nu+2n+2) a_1, \quad N_{2}(n,i)={I(n)}_{1,2} (\nu+2n+2)\tfrac{(G(n))_{2,2} }{ n+\nu+2}$$}
		
		as desired.     
	\end{proof}

	As a consequence of Propositions \ref{recurrencia xi general} and	\ref{formula j=n+i} we obtain the following result.
	\begin{thm}
		Let $n\ge 0$ and let $G(n),I(n)$ as in \eqref{GI}. 
		Then, all of the non-zero entries of $R(x,n)$ can be found in terms of $G(\ell)$ and $I(\ell)$  and the generalized Laguerre polynomials 
		$L^{(\alpha)}_{\ell}$ for $\ell=0,\ldots,n$.
	\end{thm} 
	\begin{proof}
		The equality \eqref{lagpol} implies that
		$(R(x,n))_{i,j}=L^{(\nu+j)}_{n+i-j}(x)\xi(n,i,j)$ for $n+i-j\ge 0$. 
		It is enough to show that all of the non-zero constants $\xi(n,i,j)$ can be obtained in terms of $G(\ell)$ and $I(\ell)$ for $\ell=0,\ldots,n+1$.
		
		By items ($a$)'s of Propositions \ref{recurrencia xi general} and \ref{formula j=n+i} we obtain the values of $\xi(0,i,j)$. 
		Thus, assume that $n>0$ and suppose that we want to determine $\xi(n,i,j)$ with $n+i-j>0$. 
		Notice that each time that we use items ($b$) and ($c$) of Proposition \ref{recurrencia xi general}, 
		the value of "$n+i-j$" it reduces by $1$, so if we use these items inductively we obtain that 
		$\xi(n,i,j)$ can be determined by the values of somes $\xi(n',i',j')$'s with $n'+i'-j'=0$  and in each step also appear $G(\ell)$ with $\ell=0,\ldots,n$.
		So, it is enough to see that we can determine $\xi(n,i,j)$ with $n+i-j=0$, in terms of $G(\ell)$ and  $I(\ell)$.
	\end{proof}
	
	\begin{rmk} 
		Notice that by Lemma \ref{diag GI} and Propositions \ref{recurrencia xi general}, \ref{formula j=n+i}, 
		we can replace $(G(n))_{i,i}$ and $(I(n))_{i,i+1}$ by the expressions $(\mathcal{H}_n(A^{\ast}-1)\mathcal{H}_{n-1}^{-1}))_{i,i}$ and $(\mathcal{H}_n J \mathcal{H}_n^{-1})_{i,i+1}$, respectively.
	\end{rmk}
	
	As a direct consequence of the above results and the equations obtained in Proposition \ref{prop infinitas eq}, 
	we obtain a three-terms non-linear recursion for $\mathcal{H}_{n}$. With this in mind, we need the following lemma.
	
	\begin{lem}\label{recH1}
		Let $A,J\in M_{N}(\mathbb{C})$ be matrices as in \eqref{A J}.
		Then, the matrix $\mathcal{H}_1$ can be obtained from $\mathcal{H}_0$.
	\end{lem}
	\begin{proof}
		By \eqref{fla Ad-1n}, we have that
		$$\mathcal{H}_1=\big(X(1)+[J,X(1)]\big)\mathcal{H}_0 (A^*-1)^{-1}.$$
		So, it is enough to show that we can obtain $X(1)$ in terms of $\mathcal{H}_0$.
		
		In order to show this, recall that $P(x,1)=x+ X(1)$. 
		Let us put $P_1(x):=P(x,1)$, by applying the operator $D=\partial_{x}^2 x +\partial_x ( (A-1)x +1+\nu+J) +A\nu + JA -J$ to $P_1(x)$, we have that
		$$(P_1\cdot D)(x)=  (A-1)x +1+\nu+J + (x+X(1))(A\nu + JA -J),$$
		by taking into account that $P_1\cdot D= \Gamma_1 \cdot P_1$, where $\Gamma_1$ is the discrete constant operator $A(\nu+2+J)-1-J$. 
		So we obtain that
		$$(A-1)x +1+\nu+J + (x+X(1))(A\nu + JA -J)= \big(A(\nu+2+J)-1-J\big)(x+X(1)).$$
		After some computation, we obtain that
		$$X(1)(A\nu + JA -J) +1+\nu+J = \big(A(\nu+2+J)-1-J\big)X(1),$$
		by taking into account that $[J,A]=A$, we obtain that $JA=A+AJ$ and hence
		$$[X(1),A(1+\nu+J) ] +X(1)+[J,X(1)] + 1+\nu+J -AX(1)=0.$$
		Thus, by seeing the $(i,j)$-th coordinate in the above equality we obtain the following recurrences
		{\small \begin{equation}\label{rec X1 i1}
				a_j(\nu+j+1) X(1)_{1,j+1}+ (2-j) X(1)_{1,j}+ (1+\nu+J)_{1,j}=0,
		\end{equation}}
		{\small \begin{equation}\label{rec X1 i>1}
				a_j(\nu+j+1) X(1)_{i,j+1}- a_{i-1}(\nu+i+1)X(1)_{i-1,j} + (1+i-j) X(1)_{i,j}+ (1+\nu+J)_{i,j}=0 \quad \text{for $i\ge 2$}.
		\end{equation}}
		
		Recall that by Theorem \ref{Coef R Laguerre} we have that 
		$(R(0,1))_{i,j}=0$ if $i+1< j$.
		
		\medskip
		
		{\textbf{Claim}:  If $i+1<j$ then $X(1)_{i,j}=0$.}
		
		We are going to prove this assertion by induction on $i$.
		Assume first that $i=1$ and let us consider  $j>2$,  in this case
		$$ 0=(R(0,1))_{1,j}= (K_1 P(0,1))_{1,j}=\sum_{r=1}^{N}(K_1)_{1,r} X(1)_{r,j}=X(1)_{1,j}.  $$
		
		Now, let $i>1$ and assume that the statement is true for $r<i$, that is $X(1)_{r,j}=0$ when $r+1<j$. 
		Thus, if $i+1<j$ then
		$$ 0=(R(0,1))_{i,j}= (K_1 P(0,1))_{i,j}=\sum_{r=1}^{N}(K_1)_{i,r} X(1)_{r,j}= (K_1)_{i,i}X(1)_{i,j}=X(1)_{i,j}.  $$
		By induction hypothesis and by taking into account that $K_1$ is a lower triangular matrix with $1$'s  in its diagonal, 
		we have that 
		$$\sum_{r=1}^{N}(K_1)_{i,r} X(1)_{r,j}=\sum_{r=1}^{i}(K_1)_{i,r} X(1)_{r,j}= (K_1)_{i,i}X(1)_{i,j}=X(1)_{i,j}.$$
		Therefore, $X(1)_{i,j}=0$ when $i+1<j$ as claimed.
		
		Hence, by equations \eqref{rec X1 i1} and \eqref{rec X1 i>1}, we can observe by a recursive argument, that in order to compute $X(1)$, 
		it is enough to know the value of $X(1)_{i,i+1}$ for any $i\geq 1$.
		By taking into account that 
		$$(R(0,1))_{i,i+1}=(K_1 P(0,1))_{i,i+1}=\sum_{r=1}^{i}(K_1)_{i,r} (X(1))_{r,i+1}=X(1)_{i,i+1},$$
		thus by Corollary \ref{coro R(0,n)}, we obtain that
		$$X(1)_{i,i+1}=(R(0,1))_{i,i+1}=\xi(1,i,i+1).$$
		Thus, by item ($a$) of Proposition \ref{formula j=n+i}, we obtain that 
		$$X(1)_{i,i+1}=\xi(1,i,i+1)=I(0)_{i,i+1}= (\mathcal{H}_0 J \mathcal{H}_0^{-1})_{i,i+1}.$$
	\end{proof}
	
	\begin{prop}\label{recHn}
		Let $A,J\in M_{N}(\mathbb{C})$ be matrices as in \eqref{A J}.
		Then,
		\small{\begin{align*}
				\mathcal{H}_{n+2} = (A-1)^{-2}\Big( (-[J,\mathcal{H}_n J \mathcal{H}_n^{-1}] - \mathcal{H}_n (A^{t}-1) \mathcal{H}_{n-1}^{-1} (A-1) + (A-1) \mathcal{H}_{n+1} (A^{t}-1) \mathcal{H}_n^{-1})(A-1) \\
				-2-\mathcal{H}_{n+1} J \mathcal{H}_{n+1}^{-1} + \mathcal{H}_n J \mathcal{H}_n^{-1} +(A-1)[J,\mathcal{H}_{n+1} J \mathcal{H}_{n+1}^{-1}]+(A-1)\mathcal{H}_{n+1}(A^{t}-1)\mathcal{H}_{n}^{-1}(A-1) \Big)\mathcal{H}_{n+1} (A^{t}-1)^{-1},
		\end{align*}}
		where $B_n,C_n$ and $\mathcal{H}_n$ be as in \eqref{BCH}. 
		
		Moreover,
		$$(\mathcal{H}_0)_{i,j} = \Gamma(\nu) \sum_{r=1}^{\min\{i,j\}}  \frac{\delta_r^{(\nu)}}{(i-r)!(j-r)!}\Big(\prod_{k=r}^{i-1} a_k \Big) \Big(\prod_{s=r}^{j-1} a_s\Big) (\nu)_{i+j-r},$$
		where $\Gamma$ is the Gamma function.
	\end{prop}
	\begin{proof}
		By \eqref{ML.1}, we have that
		\begin{equation}\label{eq Bn+1}
			B_{n+1} =(A-1)^{-1}\Big(B_n (A-1) - 2 - \mathcal{H}_{n+1} J \mathcal{H}_{n+1}^{-1} + \mathcal{H}_n J \mathcal{H}_n^{-1}\Big).
		\end{equation}
		On the other hand, the equation \eqref{MMd0} implies
		$$	B_n = -[J,\mathcal{H}_n J \mathcal{H}_n^{-1}] - \mathcal{H}_n (A^{t}-1) \mathcal{H}_{n-1}^{-1} (A-1) + (A-1) \mathcal{H}_{n+1} (A^{t}-1) \mathcal{H}_n^{-1}.$$
		$$B_{n+1} = -[J,\mathcal{H}_{n+1} J \mathcal{H}_{n+1}^{-1}] - \mathcal{H}_{n+1} (A^{t}-1) \mathcal{H}_{n}^{-1} (A-1) + (A-1) \mathcal{H}_{n+2} (A^{t}-1) \mathcal{H}_{n+1}^{-1}$$
		The statement it follows by changing $B_{n+1}$, and $B_n$ in \eqref{eq Bn+1}.
		
		For the last assertion, by definition we have that
		{\small $$\mathcal{H}_0= \int_{0}^{\infty} e^{xA} T^{(\nu)}(x)e^{xA^{\ast}} dx  \quad \text{ with } \quad T^{(\nu)}(x)=e^{-x} \sum_{k=1}^{N}\delta_k^{(\nu)}x^{\nu+k}E_{k,k}.$$} 
		Then,
		$$(\mathcal{H}_0)_{i,j} = \sum_{r=1}^{N}\delta_r^{(\nu)} \int_0^{\infty} e^{-x} x^{(\nu+r)} (e^{xA})_{i,r} (e^{xA^{\ast}})_{r,j}.$$
		By taking into account that $A^{\ast}_{r,j}=A_{j,r}$, and $(e^{xA})_{i,r}=\tfrac{1}{(i-r)!}\big((xA)^{i-r}\big)_{i,r}$ if $r< i$ and 
		$0$ otherwise, we have
		{\small$$ (\mathcal{H}_0)_{i,j} = \sum_{r=1}^{\min\{i,j\}}\delta_r^{(\nu)} \int_0^{\infty} e^{-x} x^{(\nu+r)} \tfrac{\big((xA)^{i-r}\big)_{i,r}}{(i-r)!} \tfrac{\big((xA)^{j-r}\big)_{j,r}}{(j-r)!} dx = \sum_{r=1}^{\min\{i,j\}}\delta_r^{(\nu)} \tfrac{(A^{i-r})_{i,r}}{(i-r)!} \tfrac{(A^{j-r})_{j,r}}{(j-r)!} \int_0^{\infty} e^{-x} x^{(\nu+i+j-r)} dx.$$}
		Notice that $(A^{i-r})_{i,r}= \prod_{k=r}^{i-1}a_k$, and recall that $\Gamma(z+1)=\int_0^{\infty} e^{-x} x^{z} dx,$ we can write 
		$$(\mathcal{H}_0)_{i,j}= \sum_{r=1}^{\min\{i,j\}} \frac{\delta_r^{(\nu)}}{(i-r)!(j-r)!}\Big(\prod_{k=r}^{i-1} a_k \Big) \Big(\prod_{s=r}^{j-1} a_s\Big) \Gamma(\nu+i+j-r+1).$$
		
		Taking into account that $\Gamma(z+1)=z\Gamma(z)$ we obtain that
		$$(\mathcal{H}_0)_{i,j} = \Gamma(\nu) \sum_{r=1}^{\min\{i,j\}}  \frac{\delta_r^{(\nu)}}{(i-r)!(j-r)!}\Big(\prod_{k=r}^{i-1} a_k \Big) \Big(\prod_{s=r}^{j-1} a_s\Big) (\nu)_{i+j-r},$$
		as asserted.
	\end{proof}
	
	\begin{rmk}
		Notice that we also have the expression 
		$$\mathcal{H}_0 = (L^{(\nu)}_{\mu}(0))^{-1} \mathcal{H}^{(\nu,\nu)}_0 (L^{(\nu)}_{\mu}(0))^{\ast},$$ 
		with $L^{(\nu)}_{\mu}(0)$ and  $\mathcal{H}^{(\nu,\nu)}_0$ defined 
		in \cite{koelink2019matrix}, where 
		$$\mu=(\mu_1, \ldots, \mu_N) \quad \text{such that } a_i=\tfrac{\mu_{i+1}}{\mu_i}\quad \text{for $i=1, \ldots, N-1$}.$$
		This assertion can be deduced from $W^{\nu}(x)= (L^{(\nu)}_{\mu}(0))^{-1} W^{(\nu,\nu)}(x) (L^{(\nu)}_{\mu}(0))^{\ast}$, where $W^{(\nu,\nu)}(x)$ are as in \cite{koelink2019matrix}. 
		They also proved that  $\mathcal{H}^{(\nu,\nu)}_0$ is a diagonal matrix.
	\end{rmk}

	The following result gives a recursion for $X(n)$'s in terms of $I(n)$.
	
	\begin{prop}
		For $n \geq 1$, let $I(n)$ as in \eqref{GI}.
		If $\delta_{i,j}$ denotes the kronecker delta function, then: 
		\begin{enumerate}[($a$)]
			\item $n \delta_{1,j} + X(n)_{1,j+1} a_j - X(n)_{1,j} = -(n+1+\nu)\delta_{1,j} - I(n)_{1,j}$.
			
			\item  $n \delta_{i,j} + X(n)_{i,j+1} a_j - a_{i-1} X(n+1)_{i-1,j} - X(n)_{i,j} + X(n+1)_{i,j} = -(n+1+\nu)\delta_{i,j} - I(n)_{i,j}$ for $i\ge 2$.
		\end{enumerate}
	\end{prop}
	
	\begin{proof}
		By Corollary \ref{Coro formula A0n}, we have that
		$$n+X(n)A - A X(n+1)-B(n) = -(n+1+\nu)-\mathcal{H}(n)J\mathcal{H}^{-1}(n).$$
		The result is a direct consequense of taking $(i,j)$-th coordinate in the above equation.
	\end{proof}
	
	\begin{rmk}
		\begin{enumerate}[($a$)]
			\item	By item $(a)$ of Corollary \ref{Coro formula A0n} we can obtain that $G(n)_{i,i}=X(n)_{i,i}$ for any $i \geq 1$.
			\item	In order to compute $P(x,n)$, for a given $n$, we can do the following:
			\begin{enumerate}
				\item Compute $\mathcal{H}_{n}$ and the explicit form of $\mathcal{H}_0,\, \mathcal{H}_1$ 
				using Lemma \ref{recH1} and Proposition \ref{recHn}.
				\item Compute $G(n)$ and $I(n)$ using $\mathcal{H}_{n}$.
				\item Compute the $\xi(n,i,j)$ using the recursions of Propositions \ref{recurrencia xi general} and \ref{formula j=n+i}. 
				\item  Then we have $R(x,n)$ and so $P(x,n)=K_n R(x,n) e^{-xA}$.
			\end{enumerate}
		\end{enumerate}
	\end{rmk}

	\section{Matrix entries of $P(x,n)$ in terms of Laguerre and dual Hahn polynomials.}
	
	In this section, we will show that under some hypothesis, the $\xi(n,i,j)$'s can be expressed in terms of dual Hahn polynomials. 
	
	\medskip
	
	\subsection*{ Some technical lemmas}
	Let $N\geq 1$ be a fixed integer and let $\mu=(\mu_1,\ldots,\mu_N)$ be a sequence of non-zero coefficients and $\alpha>0$.
	Then $L_{\mu}^{(\al)}$ is the $N\times N$ unipotent lower triangular matrix defined by
	\begin{equation}
		\label{eq:matrix-L-Laguerre}
		L^{(\alpha)}_{\mu}(x)_{m,n}=\begin{cases} \frac{\mu_{m}}{\mu_{n}} L^{(\alpha+n)}_{m-n}(x), & m\geq n,\\
			0& n<m. 
		\end{cases}
	\end{equation}
	For $\nu>0$ we consider the weight matrix
	\begin{equation}
		\label{eq:weight_factorization}
		W^{(\alpha,\nu)}_\mu(x)=L_\mu^{(\alpha)}(x)\,T^{(\nu)}(x)\,L_\mu^{(\alpha)}(x)^\ast,\qquad T^{(\nu)}(x)=e^{-x} \sum_{k=1}^N x^{\nu+k} \delta_k^{(\nu)}\, E_{k,k}.
	\end{equation}
	It can be showed that
	\begin{equation}\label{eq:decompW_1}
		W_\mu^{(\alpha, \nu)}(x) = L^{(\alpha)}_\mu(0) e^{xA_\mu} T^{(\nu)}(x) e^{xA_\mu^\ast}  L^{(\alpha)}_\mu(0)^\ast
		\quad	\text{where}\quad 
		A_{\mu}=- \sum_{k=1}^{N-1} \tfrac{\mu_{k+1}}{\mu_{k}}E_{k+1,k}.
	\end{equation}
	We impose conditions on the sequence $\{\mu_i\}_{i=1}^{N}$ and the coefficients $\delta_{k}^{(\nu)}$. 
	First of all, we assume that the coefficients $\mu_{i}$ are real and non-zero for all $i$ 
	and $\delta_k^{(\nu)}>0$, $1\leq k\leq N$, so that the weight matrix is positive definite 
	(see \cite{koelink2019matrix} for more information about the weight matrix $W^{(\alpha,\nu)}_\mu$).
	On the other hand, we consider the diagonal matrix $\Delta^{(\nu)} = \mathrm{diag}(\delta_{1}^{(\nu)}, \ldots, \delta_{N}^{(\nu)})$, so that $(T^{(\nu)})_{k,k}=e^{-x} x^{\nu+k} (\Delta^{(\nu)})_{k,k}$. We assume that there exist coefficients $c^{(\nu)}$ and $d^{(\nu)}$ such that 
	\begin{equation}
		\label{eq:condition_delta_Phi}
		\Delta^{(\nu+1)}=(d^{(\nu)} J+c^{(\nu)}) \, \Delta^{(\nu)}.
	\end{equation}
	We also assume that the coefficients $\mu_k$ and $\delta^{(\nu)}_k$ satisfy the  relation
	\begin{equation}
		\label{eq:recursion-alphas}
		\tfrac{\mu_{k+1}^2}{\mu_{k}^2}=d^{(\nu)}k(N-k) \tfrac{ \delta_{k+1}^{(\nu)} }{\delta_{k}^{(\nu+1)}}, \qquad k=1,\ldots,N-1.
	\end{equation}	
	Under the above conditions, Propositions 5.1 and 5.2 from \cite{koelink2019matrix} say that 
	\begin{equation}\label{def Phi Psi}
		\Phi^{(\alpha, \nu)}(x)=(W_\mu^{(\alpha, \nu)}(x))^{-1}W_\mu^{(\alpha, \nu+1)}(x) 
		\quad \text{and} \quad 
		\Psi^{(\alpha, \nu)}(x)=(W_\mu^{(\alpha, \nu)}(x))^{-1}\tfrac{dW_\mu^{(\alpha, \nu+1)}}{dx}(x)
	\end{equation}	
	are matrix polynomials of degree $2$ and $1$ respectively.
	Moreover, Corollary 6.3 from \cite{koelink2019matrix} asserts 
	that the operator $D_2$ defined by 
	\begin{equation}\label{def: D2}
		D_2(x)= \dfrac{d^{2}}{dx^{2}} \Phi^{\ast}(x) + \dfrac{d}{dx} \Psi^{\ast}(x)
	\end{equation}
	is symmetric respect to $W_\mu^{(\alpha, \nu)}$.

	We begin with the following technical lemma which relates the matrix polynomials $R_{n}(x)$ 
	with the constants $c^{( \nu)}$ and $d^{( \nu)}$. The proof of Lemma \ref{lema 7.1} and Lemma \ref{lema qj} can be found in the appendix.
	\begin{lem}\label{lema 7.1}
		Let $\mu=(\mu_1,\ldots,\mu_N)$ and $\delta^{(\nu)}_{k}>0$ for $1\le k\le N$, satisfying \eqref{eq:condition_delta_Phi} and \eqref{eq:recursion-alphas}.
		Let $A:=A_{\mu}$ as in \eqref{eq:decompW_1}.
		If $R(x,n)$ are the matrix polynomials defined in \eqref{def R}, then
		\begin{equation}\label{equation ev 0}
			(  R'(0,n)- R(0,n) A)C^{(\nu)}=D^{(\nu)}R(0,n)  
		\end{equation}
		where $C^{(\nu)}= (d^{(\nu)}J+c^{(\nu)})(\nu+J+1)+ ((\Delta^{(\nu)})^{-1}A \Delta^{(\nu+1)})^{\ast}$ and
		$D^{(\nu)}= n(d^{(\nu)}(J-N-1)-c^{(\nu)})$
		with $c^{( \nu)}$ and $d^{( \nu)}$ as in \eqref{eq:condition_delta_Phi}.
	\end{lem}

	In the sequel, let $0 \le n$, $1 \le i \le N$. We consider the sequence $\epsilon_{j}:=\epsilon_{j}^{(n,i)}$ defined recursively by  
	\begin{equation}\label{epsilon recursion}
		\epsilon_0=1, \quad \quad \epsilon_{j}=(n+i-j+1) d^{(\nu)}  \Big((j-1) + \tfrac{c^{(\nu)}}{d^{(\nu)}}\Big) \epsilon_{j-1} \quad \text{for $n+i-j \ge 0$}, 
	\end{equation}
	with $c^{(\nu)}, d^{(\nu)}$ defined as in \eqref{eq:condition_delta_Phi}.    
	We have the following result.
	
	\begin{lem}\label{lema epsilon's}
		Let $0\le n$ and let $1 \le N$, $1 \le i \le N$ integers and let $\{\epsilon_l\}_{l=0}^{i+n}$ be the sequence defined as in \eqref{epsilon recursion}. 
		If $\nu>0$, $\delta^{(\nu)}$ satisfies \eqref{eq:condition_delta_Phi} and $0 \le j < n+i$, we have
		$$\tfrac{\epsilon_j}{\epsilon_{j+1}}M_j=1
		\qquad \text{where} \qquad M_{j}= (n+i-j) (d^{(\nu)} j + c^{(\nu)})$$
		with $c^{(\nu)}, d^{(\nu)}$ as in \eqref{eq:condition_delta_Phi}.
	\end{lem}
	\begin{proof}
		It follows by a simple inductive argument from the definition of $\epsilon_j$.
	\end{proof}
	
	\medskip
	
	Given $i=1,\ldots,N$, $n \ge 0$ and  $j$ a positive integer such that $n+i-j \ge 0$, 
	in the sequel we consider the sequence $q_j:=q_j^{(n,i)}$ by the expression
	\begin{equation}\label{definition qj}
		q_{j}^{(n,i)} := \epsilon_{j}^{(n,i)} \xi(n,i,j),    
	\end{equation}
	where $\epsilon_j$'s are as in \eqref{epsilon recursion}.
	
	\begin{lem}
		\label{lema qj}
		Let $\mu=(\mu_1,\ldots,\mu_N)$ and $\delta^{(\nu)}_{k}>0$ for $1\le k\le N$, satisfying \eqref{eq:condition_delta_Phi} and \eqref{eq:recursion-alphas}. 
		For $1 \le i \le N$ and $0 \le n$, let $\{q_l\}_{l=0}^{i+n}$ be the sequence as in \eqref{definition qj}.
		Then, the sequence $\{q_l\}_{l=0}^{i+n}$ satisfies
		$$ E_j q_j +  F_j  q_{j-1} + \tfrac{1}{d^{(\nu)}}   q_{j+1} = 0, \quad n+i-j > 0,$$
		where 
		\begin{align*}
			E_j &=(n+i-j) ( j + \tfrac{c^{(\nu)}}{d^{(\nu)}}) +  (j-1) (N-j+1) + n  ( (i-N-1)-\tfrac{c^{(\nu)}}{d^{(\nu)}}) ,\\
			F_j &=(j-1) (N-j+1) \tfrac{\mu_{j-1}}{\mu_j}  (n+i-j+1)  \big(d^{(\nu)} (j-1) + c^{(\nu)} \big)    
		\end{align*}
		with $c^{(\nu)}, d^{(\nu)}$ as in \eqref{eq:condition_delta_Phi}.
	\end{lem}
	
	\medskip
	
	\subsection*{Dual Hahn polynomials}
	
	For $0 \le n$, $1 \le i \le N$ let us consider the following sequence 
	\begin{equation}\label{definition qtj}
		\widetilde{q_j}^{(n,i)}:= {(d^{(\nu)})}^{-j} q_j^{(n,i)}, \quad n+i-j > 0,
	\end{equation}where $q_j$'s are as in \eqref{definition qj} and $d^{(\nu)}$ as \eqref{eq:condition_delta_Phi}.

	\begin{lem}\label{lema qtilde}
		Let $\mu=(\mu_1,\ldots,\mu_N)$ and $0 < \delta^{(\nu)}_{k}$, for $1\le k\le N$, satisfying \eqref{eq:condition_delta_Phi} and \eqref{eq:recursion-alphas}. 
		Let $\widetilde{q_l}$ be  as in \eqref{definition qtj}.
		If $\mu_j=1$ for all $j$ then the $\widetilde{q_l}$'s satisfy the following equation
		$$\widetilde{E}_j \widetilde{q_j}+ \widetilde{F}_j  \widetilde{q}_{j-1} +  \widetilde{q}_{j+1} = 0, \quad n+i-j > 0,$$
		where
		{\small \begin{align*}
				\widetilde{E}_j &=(n+i-j) ( j + \tfrac{c^{(\nu)}}{d^{(\nu)}}) +  (j-1) (N-j+1) + n  ( i-N-1-\tfrac{c^{(\nu)}}{d^{(\nu)}}), \\
				\widetilde{F}_j &= (j-1) (N-j+1)  (n+i-j+1)  \Big(j-1 + \tfrac{c^{(\nu)}}{d^{(\nu)}} \Big)
		\end{align*}}
		with $c^{(\nu)}, d^{(\nu)}$ as in \eqref{eq:condition_delta_Phi}.
	\end{lem}
	\begin{proof}
		By Lemma \ref{lema qj}, we have that $ E_j q_j +  F_j  q_{j-1} + \tfrac{1}{d^{(\nu)}}   q_{j+1} = 0$
		with 
		{\small \begin{align*}
				E_j &=(n+i-j) ( j + \tfrac{c^{(\nu)}}{d^{(\nu)}}) +  (j-1) (N-j+1) + n  ( (i-N-1)-\tfrac{c^{(\nu)}}{d^{(\nu)}}) ,\\
				F_j &=(j-1) (N-j+1) \tfrac{\mu_{j-1}}{\mu_j}  (n+i-j+1)  \big(d^{(\nu)} (j-1) + c^{(\nu)} \big).    
		\end{align*}}
		Now, since $\mu_j=1$ for all $j$ and  $\widetilde{q_j}(x):= {(d^{(\nu)})}^{-j} q_j(x)$,
		we have that
		{\small \begin{eqnarray*}
				&& \Big((n+i-j) ( j + \tfrac{c^{(\nu)}}{d^{(\nu)}}) +  (j-1) (N-j+1) + n  ( (i-N-1)-\tfrac{c^{(\nu)}}{d^{(\nu)}}) \Big) \widetilde{q_j} {(d^{(\nu)})}^{j} \\
				&+&  (j-1) (N-j+1)  (n+i-j+1)  \big(d^{(\nu)} (j-1) + c^{(\nu)} \big) {(d^{(\nu)})}^{j-1} \widetilde{q}_{j-1} + \tfrac{1}{d^{(\nu)}} {(d^{(\nu)})}^{j+1}  \widetilde{q}_{j+1} = 0.
		\end{eqnarray*}}
		By multiplication for ${(d^{(\nu)})}^{-j}$, we obtain that
		$$\widetilde{E}_j \widetilde{q_j}+ \widetilde{F}_j  \widetilde{q}_{j-1} +  \widetilde{q}_{j+1} = 0,$$
		with
		{\small \begin{align*}
				\widetilde{E}_j &=(n+i-j) ( j + \tfrac{c^{(\nu)}}{d^{(\nu)}}) +  (j-1) (N-j+1) + n  ( i-N-1-\tfrac{c^{(\nu)}}{d^{(\nu)}}), \\
				\widetilde{F}_j &= (j-1) (N-j+1)  (n+i-j+1)  \Big(j-1 + \tfrac{c^{(\nu)}}{d^{(\nu)}} \Big),
		\end{align*}}
		as asserted.
	\end{proof}
	
	\medskip 
	
	Recall that if the sequence of polynomials  $\{s_k\}$ satisfies the normalized recurrence relations 
	\begin{equation}\label{recurrence sn}
		xs_k(x)= s_{k+1}(x) -(u_k+v_k) s_k(x)+u_{k-1}v_{k} s_{k-1}(x),
	\end{equation}
	with 
	$$u_k=(k+\gamma+1)(k-M), \qquad v_k=k(k-\delta-M-1),$$
	then $s_k$ satisfies 
	$$T_k(\lambda(x);\gamma,\delta,M) = \tfrac{1}{(\gamma+1)_{k} (-M)_{k}} s_k(\lambda(x))$$ 
	where $\lambda(x)= x(x+\gamma+\delta +1)$ and $\{T_k\}$ is the sequence of dual Hahn polynomials defined by
	$$T_k(\lambda(x);\gamma,\delta,M)=  \setlength\arraycolsep{1.1pt}
	{}_3 F_2\left(\begin{matrix}-k,& &-x,& &x+\gamma+\delta+1\\[.1em] 
		&&\gamma+1,&&-M&\end{matrix};1\right) \qquad \text{for } k=0,1,\ldots,M. $$
	\begin{prop}\label{prop: dH qtilde}
		Let $\mu=(\mu_1,\ldots,\mu_N)$ and $\delta^{(\nu)}_{k}>0$ for $1\le k\le N$, satisfying \eqref{eq:condition_delta_Phi} and \eqref{eq:recursion-alphas}. 
		Let $\widetilde{q_j}$ be  as in \eqref{definition qtj}.
		If $\mu_{\ell}=1$ for all $\ell=1,\ldots,N$, then the $\widetilde{q_j}$'s satisfy
		\begin{equation}\label{eq: qtilde dH}
			\widetilde{q}_{j}= (\gamma+1)_{j-1}(-(N-1))_{j-1} T_{j-1}(\lambda(x^{(n,i)}); \gamma,\delta, N-1) \quad n+i-j > 0
		\end{equation}
		where $\gamma= \tfrac{c^{(\nu)}}{d^{(\nu)}}$,  $\delta=n+i-N$ and $x^{(n,i)}=(\gamma+1)(N+i-2)-n(N-i)$.
	\end{prop}
	\begin{proof}
		By Lemma \ref{lema qtilde}, we have that
		$$\widetilde{E}_j \widetilde{q_j}+ \widetilde{F}_j  \widetilde{q}_{j-1} +  \widetilde{q}_{j+1} = 0, \quad n+i-j > 0,$$
		where
		{\small \begin{align*}
				\widetilde{E}_j &=(n+i-j) ( j + \tfrac{c^{(\nu)}}{d^{(\nu)}}) +  (j-1) (N-j+1) + n  ( i-N-1-\tfrac{c^{(\nu)}}{d^{(\nu)}}), \\
				\widetilde{F}_j &= (j-1) (N-j+1)  (n+i-j+1)  \Big(j-1 + \tfrac{c^{(\nu)}}{d^{(\nu)}} \Big)
		\end{align*}}
		with $c^{(\nu)}, d^{(\nu)}$ as in \eqref{eq:condition_delta_Phi}. Now, if we consider
		$$u_k=(k+\gamma+1)(k-M)\text{ and } v_k=k(k-\delta-M-1),$$
		with $\gamma= \tfrac{c^{(\nu)}}{d^{(\nu)}}$, $\delta=n+i-N$ and $M=N-1$. It is straightforward to check that
		$$\widetilde{F}_j = u_{j-2} v_{j-1} \quad \text{and} \quad \widetilde{E}_j = -(u_{j-1}+v_{j-1}+x^{(n,i)})$$	
		where $x^{(n,i)}=(\gamma+1)(N+i-2)-n(N-i)$. 
		Thus, since $\widetilde{q_j}$ can be seen as constants polynomial, 
		we have that
		$$x^{(n,i)}\widetilde{q_j}(x^{(n,i)}) = 
		\widetilde{q}_{j+1}(x^{(n,i)}) -(u_{j-1}+v_{j-1}) \widetilde{q}_j(x^{(n,i)})+u_{j-2} v_{j-1} \widetilde{q}_{j-1}(x^{(n,i)}).$$
		Therefore, by definition of dual Hahn polynomials and by taking into account that $\widetilde{q}_{j}$ is a constant polynomial, we obtain \eqref{eq: qtilde dH} as desired.
	\end{proof}

	\medskip
	
	By recalling that $\widetilde{q}_{j}=(d^{(\nu)})^{-j} q_j = \epsilon_{j} (d^{(\nu)})^{-j} \xi(n,i,j)$ for $n+i-j > 0$, we obtain the following result
	
	\begin{thm}
		Let $\mu=(\mu_1,\ldots,\mu_N)$ and $\delta^{(\nu)}_{k} > 0$, for $1\le k\le N$, satisfying \eqref{eq:condition_delta_Phi} and \eqref{eq:recursion-alphas}. Let $\epsilon_{j}$ as in \eqref{epsilon recursion}.
		If $\mu_{\ell}=1$ for all $\ell=1,\ldots,N$, then the constants $\xi(n,i,j)$'s satisfy
		\begin{equation}\label{eq: xi dH}
			\xi(n,i,j)= \tfrac{(d^{(\nu)})^{j}(\gamma+1)_{j-1}(-(N-1))_{j-1}}{\epsilon_{j}} T_{j-1}(\lambda(x^{(n,i)}); \gamma,\delta, N-1), \quad n+i-j > 0
		\end{equation}
		where $\gamma= \tfrac{c^{(\nu)}}{d^{(\nu)}}$,  $\delta=n+i-N$ and $x^{(n,i)}=(\gamma+1)(N+i-2)-n(N-i)$.
		Moreover, if $n+i=j$, the constants $\xi(n,i,j)$'s satisfy the recursion	
		\begin{equation}
			\xi(0,1,1)=\tfrac{1}{\nu +2}, \qquad	  \xi(n,i,j-1) =
			\Big( \tfrac{n(N+1-i)(i-1)}{(j-1) (N-j+1)} + 1 \Big)\xi(n,i,j), \quad j >1,
		\end{equation}
		
	\end{thm}
	\begin{proof}
		If $n+i-j > 0$, by Proposition \ref{prop: dH qtilde} we have that 
		$$\xi(n,i,j)=\tfrac{(d^{(\nu)})^j}{\epsilon_j} \widetilde{q}_{j}= \tfrac{(d^{(\nu)})^{j}(\gamma+1)_{j-1}(-(N-1))_{j-1}}{\epsilon_{j}} T_{j-1}(\lambda(x^{(n,i)}); \gamma,\delta, N-1) $$
		with $\gamma= \tfrac{c^{(\nu)}}{d^{(\nu)}}$,  $\delta=n+i-N$ and $x^{(n,i)}=(\gamma+1)(N+i-2)-n(N-i)$ as asserted.	
		
		\medskip
		
		Now, if we take $j=n+i$ in the expression \eqref{equation ev 0}, we obtain
		\begin{eqnarray*}
			& &R_n'(0)_{i,j} (d^{(\nu)} j + c^{(\nu)}) (\nu + j +1)  + {R_n'(0)}_{i(j-1)} {((\Delta^{(\nu)})^{-1}A \Delta^{(\nu+1)})^{\ast}}_{(j-1)j}\\
			& & - ({R_n(0)}_{i,j} A_{j(j-1)}{((\Delta^{(\nu)})^{-1}A \Delta^{(\nu+1)})^{\ast}}_{(j-1),j} = n(d^{(\nu)}(J-N-1)-c^{(\nu)}) R(n,0)_{i,j}.
		\end{eqnarray*}

		By taking into account that $R_{n}(x)_{i,j}=L^{(\nu+j)}_{n+i-j}(x) \xi(n,i,j)$, 
		{\small $$L^{\alpha}_n(0)=\tfrac{(\alpha+1)_n}{n!} \qquad  \text{and}
			\qquad \tfrac{\partial}{\partial x}L^{\alpha}_n(x)=(-1) L^{\alpha+1}_{n-1}(x),$$}
		we have
		{\small \begin{eqnarray*}
				- \xi(n,i,j-1) {((\Delta^{(\nu)})^{-1}A \Delta^{(\nu+1)})^{\ast}}_{(j-1)j} &-& (\xi(n,i,j) A_{j(j-1)}{((\Delta^{(\nu)})^{-1}A \Delta^{(\nu+1)})^{\ast}}_{(j-1),j} \\
				&=& n(d^{(\nu)}(i-N-1)-c^{(\nu)}) \xi(n,i,j).
		\end{eqnarray*}}
		Recall that by \eqref{eq:recursion-alphas} we have that
		{\small\begin{equation*}
				{{({(\Delta^{(\nu)})}^{-1} A \Delta^{(\nu+1)} \big)}^{\ast}}_{(j-1,j)} = d^{(\nu)} (j-1) (N-j+1) \tfrac{\mu_{j-1}}{\mu_j},
		\end{equation*}}
		then, taking in account that $\mu_k=1$ for all $k$, we have
		{\small \begin{eqnarray*}
				- \xi(n,i,j-1) d^{(\nu)} (j-1) (N-j+1)  &-& (\xi(n,i,j) A_{j(j-1)}d^{(\nu)} (j-1) (N-j+1) \\
				&=& n(d^{(\nu)}(i-N-1)-c^{(\nu)}) \xi(n,i,j).
		\end{eqnarray*}}
		
		Thus, using the conditions \eqref{eq:condition_delta_Phi}, \eqref{eq:recursion-alphas}, and taking in account that $A_{s,s-1}=-1$, we obtain
		{\small \begin{equation*}
				- d^{(\nu)} (j-1) (N-j+1) \xi(n,i,j-1) =
				\Big(-nd^{(\nu)}(N+1-i)(i-1) - d^{(\nu)} (j-1) (N-j+1) \Big)\xi(n,i,j),
		\end{equation*}}
		as desired.
		
		The last assertion it follows from $$\xi(0,1,1)(\nu +2)=\xi(0,1,1) L^{1}_1(0)=R(0,0)_{1,1}=(K_1 P(0,0))_{1,1}=1$$
		
	\end{proof}

	\bibliographystyle{plain}
	\bibliography{biblio}		

\begin{thebibliography}{10}

\bibitem{Arniz-2014}
Gerardo Ariznabarreta and Manuel Ma\~nas.
\newblock Matrix orthogonal {L}aurent polynomials on the unit circle and {T}oda
  type integrable systems.
\newblock {\em Adv. Math.}, 264:396--463, 2014.

\bibitem{Boch}
S.~Bochner.
\newblock {\"U}ber {S}turm-{L}iouvillesche {P}olynomsysteme.
\newblock {\em Math. Z.}, 29(1):730--736, 1929.

\bibitem{Casper2}
W.R. Casper and M.~Yakimov.
\newblock The matrix {Bochner} problem.
\newblock {\em American Journal of Mathematics}, 144(4):1009--1065, 2022.

\bibitem{DamanikPS}
David Damanik, Alexander Pushnitski, and Barry Simon.
\newblock The analytic theory of matrix orthogonal polynomials.
\newblock {\em Surv. Approx. Theory}, 4:1--85, 2008.

\bibitem{DERom}
Alfredo Deaño, Bruno Eijsvoogel, and Pablo Román.
\newblock Ladder relations for a class of matrix valued orthogonal polynomials.
\newblock {\em Studies in Applied Mathematics}, 146(2):463--497, 2021.

\bibitem{DuitsK}
Maurice Duits and Arno~B.J. Kuijlaars.
\newblock The two periodic aztec diamond and matrix valued orthogonal
  polynomials.
\newblock to appear in Journal of the European Mathematical Society, 2017.

\bibitem{Duran1}
Antonio~J. Dur\'an.
\newblock Matrix inner product having a matrix symmetric second order
  differential operator.
\newblock {\em Rocky Mountain J. Math.}, 27(2):585--600, 1997.

\bibitem{Duran2009_2}
Antonio~J. Dur{\'a}n.
\newblock A method to find weight matrices having symmetric second-order
  differential operators with matrix leading coefficient.
\newblock {\em Constr. Approx.}, 29(2):181--205, 2009.

\bibitem{DuranG1}
Antonio~J. Dur{\'a}n and F.~Alberto Gr{\"u}nbaum.
\newblock Orthogonal matrix polynomials satisfying second-order differential
  equations.
\newblock {\em Int. Math. Res. Not.}, 2004(10):461--484, 2004.

\bibitem{Geronimo}
J.~S. Geronimo.
\newblock Scattering theory and matrix orthogonal polynomials on the real line.
\newblock {\em Circuits Systems Signal Process.}, 1(3-4):471--495, 1982.

\bibitem{gorbatsevich1998level}
V.~V. Gorbatsevich.
\newblock On the level of some solvable lie algebras.
\newblock {\em Siberian Mathematical Journal}, 39(5):872--883, 1998.

\bibitem{GroeneveltIK}
Wolter Groenevelt, Mourad E.~H. Ismail, and Erik Koelink.
\newblock Spectral decomposition and matrix-valued orthogonal polynomials.
\newblock {\em Adv. Math.}, 244:91--105, 2013.

\bibitem{IglesiaG}
F.~Alberto Gr\"{u}nbaum and Manuel~D. de~la Iglesia.
\newblock Matrix valued orthogonal polynomials arising from group
  representation theory and a family of quasi-birth-and-death processes.
\newblock {\em SIAM J. Matrix Anal. Appl.}, 30(2):741--761, 2008.

\bibitem{GdIM}
F.~Alberto Gr{\"u}nbaum, Manuel~D. de~la Iglesia, and Andrei
  Mart\'{\i}nez-Finkelshtein.
\newblock Properties of matrix orthogonal polynomials via their
  {R}iemann-{H}ilbert characterization.
\newblock {\em SIGMA Symmetry Integrability Geom. Methods Appl.}, 7(098):31
  pages, 2011.

\bibitem{HeckmanP}
Gert Heckman and Maarten van Pruijssen.
\newblock Matrix valued orthogonal polynomials for {G}elfand pairs of rank one.
\newblock {\em Tohoku Math. J. (2)}, 68(3):407--437, 2016.

\bibitem{Ismail}
Mourad E.~H. Ismail.
\newblock Classical and quantum orthogonal polynomials in one variable.
\newblock 98, 2005.

\bibitem{ISMAIL2019235}
Mourad~E.H. Ismail, Erik Koelink, and Pablo Román.
\newblock Matrix valued hermite polynomials, burchnall formulas and non-abelian
  toda lattice.
\newblock {\em Advances in Applied Mathematics}, 110:235--269, 2019.

\bibitem{knapp1996lie}
Anthony~W. Knapp.
\newblock {\em Lie groups beyond an introduction}, volume 140.
\newblock Springer, 1996.

\bibitem{KdlRR}
Erik Koelink, Ana~M. de~los R\'{\i}os, and Pablo Rom{\'a}n.
\newblock Matrix-valued {G}egenbauer-type polynomials.
\newblock {\em Constr. Approx.}, 46(3):459--487, 2017.

\bibitem{KR}
Erik Koelink and Pablo Rom{\'a}n.
\newblock Orthogonal vs. non-orthogonal reducibility of matrix-valued measures.
\newblock {\em SIGMA Symmetry Integrability Geom. Methods Appl.}, 12(008):9
  pages, 2016.

\bibitem{koelink2019matrix}
Erik Koelink and Pablo Rom{\'a}n.
\newblock Matrix valued laguerre polynomials.
\newblock {\em Positivity and Noncommutative Analysis: Festschrift in Honour of
  Ben de Pagter on the Occasion of his 65th Birthday}, pages 295--320, 2019.

\bibitem{KvPR1}
Erik Koelink, Maarten van Pruijssen, and Pablo Rom{\'a}n.
\newblock Matrix-valued orthogonal polynomials related to
  $(\mathrm{SU}(2)\times\mathrm{SU}(2),\mathrm{diag})$.
\newblock {\em Int. Math. Res. Not. IMRN}, 2012(24):5673--5730, 2012.

\bibitem{KvPR2}
Erik Koelink, Maarten van Pruijssen, and Pablo Rom{\'a}n.
\newblock Matrix-valued orthogonal polynomials related to
  $(\mathrm{SU}(2)\times\mathrm{SU}(2),\mathrm{diag})$, {II}.
\newblock {\em Publ. Res. Inst. Math. Sci.}, 49(2):271--312, 2013.

\bibitem{Koornwinder}
Tom~H. Koornwinder.
\newblock Matrix elements of irreducible representations of {${\rm
  SU}(2)\times{\rm SU}(2)$} and vector-valued orthogonal polynomials.
\newblock {\em SIAM J. Math. Anal.}, 16(3):602--613, 1985.

\bibitem{PATERA19901}
J.~Patera and H.~Zassenhaus.
\newblock Solvable lie algebras of dimension $\le 4$ over perfect fields.
\newblock {\em Linear Algebra and its Applications}, 142:1--17, 1990.

\end{thebibliography}
	
	\appendix\section{}
	
	In this apprendix we give the proofs of the  Lemmas \ref{lema 7.1} and \ref{lema qj}.
	
	\begin{proof}[Proof of Lemma \ref{lema 7.1}]
		Let $P_{n}(x)=P(x,n)$ be the sequence of monic orthogonal polynomials 
		respect to the weight $W^{(\nu)}$ as in \eqref{def: Wnu}.
		Since $W_\mu^{(\alpha, \nu)}(x)=L_{\mu}^{(\alpha)}(0) W^{(\nu)}(x) L_{\mu}^{(\alpha)}(0)^*$ we have that 
		{\small $$\int_0^{\infty} P_n(x) (L_{\mu}^{(\alpha)}(0)^{-1})W_\mu^{(\alpha, \nu)}(x) (P_{m}(x)L_{\mu}^{(\alpha)}(0)^{-1})^* dx =\int_0^{\infty} P_n(x) W^{(\nu)}(x) P_{m}(x)^* dx= \delta_{n,m}\mathcal{H}_{n}.$$}
		Hence, if $P_{n}^{(\alpha,\nu)}(x)=L_{\mu}^{(\alpha)}(0) P_n(x) L_{\mu}^{(\alpha)}(0)^{-1} $ then  
		$P^{(\alpha,\nu)}_n(x)$ is the sequence of monic orthogonal polynomials with respect to the weight $W_\mu^{(\alpha, \nu)}$.
		Now, if $D_2$ is the operator defined in \eqref{def: D2}, the Corollary 6.3 of \cite{koelink2019matrix} implies that
		$P^{(\alpha,\nu)}_n D_2= n {\hat{K}_n}^{(\alpha,\nu)} P^{(\alpha,\nu)}_{n}$ 
		for certain matrix ${\hat{K}_n}^{(\alpha,\nu)}$ defined recursively in section 6 from \cite{koelink2019matrix}, and so 
		{\small \begin{equation}\label{equation PKhat}
				P_n L_{\mu}^{(\alpha)}(0)^{-1}D_2 L_{\mu}^{(\alpha)}(0)=nL_{\mu}^{(\alpha)}(0)^{-1}{\hat{K}_n}^{(\alpha,\nu)} L_{\mu}^{(\alpha)}(0) P_n.
		\end{equation}}
		By Proposition 6.1 of \cite{koelink2019matrix}, we obtain that
		{\small $${(L_{\mu}^{(\alpha)}(0))}^{-1}{\hat{K}_n}^{(\alpha,\nu)} L_{\mu}^{(\alpha)}(0)=d^{(\nu)}(J-(J+\nu+n)A-N-1)-c^{(\nu)},$$ }
		and by taking into account that $JA-AJ=A$, we have that
		{\small $$J-JA-(\nu+n)A-N-1=J-AJ-A-(\nu+n)A-N-1.$$}
		On the other hand, since $\Gamma_{n}=A(n+\nu+J+1)-(n+J)=-J+(n+\nu)A+A+AJ-n$, then
		{\small $${(L_{\mu}^{(\alpha)}(0)})^{-1}{\hat{K}_n}^{(\alpha,\nu)} L_{\mu}^{(\alpha)}(0)=d^{(\nu)}(-\Gamma_{n}-n-N-1)-c^{(\nu)}.$$}
		By \eqref{equation PKhat}, we have that
		{\small $$K_n R_n e^{-xA} {(L_{\mu}^{(\alpha)}(0))}^{-1}D_2 L_{\mu}^{(\alpha)}(0)= n {(L_{\mu}^{(\alpha)}(0))}^{-1}{\hat{K}_n}^{(\alpha,\nu)} L_{\mu}^{(\alpha)}(0)K_n R_n e^{-xA}$$}
		and thus
		{\small $$R_n e^{-xA} {(L_{\mu}^{(\alpha)}(0))}^{-1}D_2 L_{\mu}^{(\alpha)}(0)e^{xA}=n K_{n}^{-1}(d^{(\nu)}(-\Gamma_n-n-N-1)-c^{(\nu)})K_n R_n.$$}
		By recalling that $K_{n}^{-1}\Gamma_n K_n =-n-J$, we have that
		{\small \begin{equation}\label{equation R D2}
				R_n e^{-xA} {(L_{\mu}^{(\alpha)}(0))}^{-1}D_2 L_{\mu}^{(\alpha)}(0)e^{xA}=n(d^{(\nu)}(J-N-1)-c^{(\nu)})R_{n}.
		\end{equation}}
		Now, by taking into account that $D_2= \dfrac{d^{2}}{dx^{2}} \Phi^{\ast}(x) + \dfrac{d}{dx} \Psi^{\ast}(x)$ with $\Phi, \Psi$ polynomials of degree $2$ and $1$. In general, if $U(x)$ is a polynomial, we have that
		{\small \begin{eqnarray*}
				U(x) \cdot e^{-xA} {(L_{\mu}^{(\alpha)}(0))}^{-1}D_2 L_{\mu}^{(\alpha)}(0)e^{xA} & = & \dfrac{d^{2}}{dx^{2}} \Big( U(x) e^{-xA}\Big) {(L_{\mu}^{(\alpha)}(0))}^{-1} \Phi^{\ast}(x) L_{\mu}^{(\alpha)}(0)e^{xA}\\
				&& + \dfrac{d}{dx} \Big( U(x) e^{-xA}\Big) {(L_{\mu}^{(\alpha)}(0))}^{-1} \Psi^{\ast}(x) L_{\mu}^{(\alpha)}(0)e^{xA} \\
				& = & \Big( U''(x)-2 U'(x)A+ U(x) A^{2}\Big) e^{-xA} {(L_{\mu}^{(\alpha)}0))}^{-1} \Phi^{\ast}(x) L_{\mu}^{(\alpha)}(0)e^{xA} \\
				& & + \Big(  U'(x)- U(x) A\Big) e^{-xA} {(L_{\mu}^{(\alpha)}(0))}^{-1} \Psi^{\ast}(x) L_{\mu}^{(\alpha)}(0)e^{xA}.
		\end{eqnarray*}}
		Hence, we want to find some easy expression for
		$$e^{-xA} L(0)^{-1} \Phi^{\ast}(x) L(0) e^{xA}\quad \text{and} \quad e^{-xA} L(0)^{-1} \Psi^{\ast}(x) L(0) e^{xA}.$$
		Some similar expressions was studied by Koelink and Roman (see \cite{koelink2019matrix}).
		By Corollary 5.3 in \cite{koelink2019matrix}, we have that
		
		{ \small\begin{eqnarray*}
				L_{\mu}^{(\alpha)}(0)^{\ast} \Phi(x) ((L_{\mu}^{(\alpha)}(0))^{\ast})^{-1} &= & -d^{(\nu)}x^2 A^{\ast}+x (d^{(\nu)} J +c^{(\nu)})\\
				L_{\mu}^{(\alpha)}(0)^{\ast} \Psi(x) ((L_{\mu}^{(\alpha)}(0))^{\ast})^{-1}& = & x\Big(d^{(\nu)}(J-A^{\ast}(J+\nu+1) -N-1)-c^{(\nu)}) \Big) \\
				&  & + (\nu+J+1)(d^{(\nu)}J+c^{(\nu)})+(\Delta^{(\nu)})^{-1}A \Delta^{(\nu+1)}. 
		\end{eqnarray*}}
		Hence, we have that
		\small{\begin{eqnarray*}
				e^{-xA}L_{\mu}^{(\alpha)}(0)^{-1} \Phi^{\ast}(x)L_{\mu}^{(\alpha)}(0) e^{xA} &=& -d^{(\nu)}x^2 e^{-xA}A e^{xA}+xd^{(\nu)}e^{-xA}Je^{xA}+x c^{(\nu)} \\
				&=& -d^{(\nu)}x^2 A +xd^{(\nu)}(xA+J)+ xc^{(\nu)} \\
				&=& x(d^{(\nu)}J+c^{(\nu)}).
		\end{eqnarray*}}
		{\small \begin{eqnarray*}
				e^{-xA}L_{\mu}^{(\alpha)}(0)^{-1} \Psi^{\ast}(x)L_{\mu}^{(\alpha)}(0) e^{xA} &=& x\Big(d^{(\nu)}(xA+J-(xA+J+\nu+1)A-N-1)-c^{(\nu)}) \Big) \\
				& & + (d^{(\nu)}(xA+J)+c^{(\nu)})(\nu+xA+J+1)+e^{-xA} ((\Delta^{(\nu)})^{-1}A \Delta^{(\nu+1)})^{\ast} e^{xA}.
		\end{eqnarray*}}
		By evaluation in $x=0$ in \eqref{equation R D2} and by taking into account the above expressions, 
		we obtain \eqref{equation ev 0}, as desired.
	\end{proof}

	\begin{proof}[Proof of Lemma \ref{lema qj}]
		By taking into account that $(R_{n}(x))_{i,j}$ satisfies the expression given by Theorem \ref{Coef R Laguerre}, the $(i,j)$-coordinate of the matrix in the right hand of \eqref{equation ev 0} is 
		{\small $$ n(d^{(\nu)}(i-N-1)-c^{(\nu)}){(R_{n}(0))}_{i,j}= n(d^{(\nu)}(i-N-1)-c^{(\nu)})L^{(\nu+j)}_{n+i-j}(0) \xi(n,i,j).$$}
		On the other hand, for $n+i-j > 0$ the $(i,j)$-coordinate of the matrix in the left hand of \eqref{equation ev 0} is
		{\small \begin{eqnarray*}
				&& R_n'(0)_{i,j} (d^{(\nu)} j + c^{(\nu)}) (\nu + j +1)  + {R_n'(0)}_{i(j-1)} {((\Delta^{(\nu)})^{-1}A \Delta^{(\nu+1)})^{\ast}}_{(j-1)j}\\
				& & - {R_n(0)}_{i(j+1)} (d^{(\nu)} j + c^{(\nu)}) (\nu + j +1) A_{(j+1)j}  - ({R_n(0)}_{i,j} A_{j(j-1)}{((\Delta^{(\nu)})^{-1}A \Delta^{(\nu+1)})^{\ast}}_{(j-1),j}.
		\end{eqnarray*}}
		By taking into account that $R_{n}(x)_{i,j}=L^{(\nu+j)}_{n+i-j}(x) \xi(n,i,j)$, 
		{\small $$L^{\alpha}_n(0)=\tfrac{(\alpha+1)_n}{n!} \qquad  \text{and}
			\qquad \tfrac{\partial}{\partial x}L^{\alpha}_n(x)=(-1) L^{\alpha+1}_{n-1}(x),$$}
		we obtain that the above expression is equivalent to
		{\tiny \begin{eqnarray*}
				- n \Big(d^{(\nu)} (i-N-1)-c^{(\nu)}\Big) \tfrac{{(\nu+j+1)}_{n+i-j}}{(n+i-j)!} \xi(n,i,j) &=& \tfrac{{(\nu+j+2)}_{n+i-j-1}}{(n+i-j-1)!} (d^{(\nu)} j + c^{(\nu)}) (\nu +j +1) \xi(n,i,j)\\
				&& + \tfrac{{(\nu+j+1)}_{n+i-j}}{(n+i-j)!} {{\big( {(\Delta^{(\nu)})}^{-1} A \Delta^{(\nu+1)} \big)}^{\ast}}_{(j-1,j)} \xi(n,i,j-1)\\
				&& + \tfrac{{(\nu+j+2)}_{n+i-j-1}}{(n+i-j-1)!} (d^{(\nu)} j + c^{(\nu)}) (\nu+j+1) \xi(n,i,j+1)\\
				&& + \tfrac{{(\nu+j+1)}_{(n+i-j)}}{(n+i-j)!} A_{j(j-1)} {{\Big({(\Delta^{(\nu)})}^{-1} A \Delta^{(\nu)}\Big)}^{\ast}}_{(j-1)j} \xi(n,i,j).
		\end{eqnarray*}}
		Thus, by taking into account that $(\nu+j+2)_{n+i-j-1}(\nu +j +1)= (\nu+j+1)_{n+i-j-1}$, and  multiplying both sides by $\tfrac{(n+i-j-1)!}{(\nu+j+1)_{n+i-j-1}}$, we obtain 
		{\small \begin{eqnarray*}
				- n \Big(d^{(\nu)} (i-N-1)-c^{(\nu)}\Big) \tfrac{(n+i-j-1)}{n+i-j} \xi(n,i,j)& =&  (d^{(\nu)} j + c^{(\nu)}) \xi(n,i,j)\\
				&&+   \tfrac{(n+i-j-1)}{n+i-j} {{\big( {(\Delta^{(\nu)})}^{-1} A \Delta^{(\nu+1)} \big)}^{\ast}}_{(j-1,j)} \xi(n,i,j-1)\\
				&& +   (d^{(\nu)} j + c^{(\nu)}) \xi(n,i,j+1)\\
				& &+ \tfrac{(n+i-j-1)}{n+i-j} A_{j(j-1)} {{\Big({(\Delta^{(\nu)})}^{-1} A \Delta^{(\nu)}\Big)}^{\ast}}_{(j-1)j} \xi(n,i,j).\\ 
		\end{eqnarray*}}
		Finally, we obtain that 
		{\small \begin{eqnarray*}
				& & \Big((n+i-j)(d^{(\nu)} j + c^{(\nu)}) + A_{j(j-1)} {{\big({(\Delta^{(\nu)})}^{-1} A \Delta^{(\nu)}\big)}^{\ast}}_{(j-1)j} + n (d^{(\nu)} (i-N-1)-c^{(\nu)}) \Big) \xi(n,i,j)\\
				& & + {{\big( {(\Delta^{(\nu)})}^{-1} A \Delta^{(\nu+1)} \big)}^{\ast}}_{(j-1,j)} \xi(n,i,j-1) +  (n+i-j) (d^{(\nu)} j + c^{(\nu)}) \xi(n,i,j+1) = 0.
		\end{eqnarray*}}
		By \eqref{eq:recursion-alphas} we have that $\tfrac{{\delta^{(\nu)}_{k+1}}}{{\delta^{(\nu+1)}_k}} = \tfrac{\mu_{k+1}^2}{d^{(\nu)} k (N-k) \mu_{k}^2}$,
		and then
		{\small\begin{eqnarray*}
				{{({(\Delta^{(\nu)})}^{-1} A \Delta^{(\nu+1)} \big)}^{\ast}}_{(j-1,j)} &=& \tfrac{{\delta^{(\nu+1)}_{j-1}}}{\delta^{(\nu)}_j} A_{j(j-1)} \\
				&=& d^{(\nu)} (j-1) (N-j+1) \tfrac{{\mu_{j-1}}^{2}}{{\mu_j}^{2}} \tfrac{\mu_j}{\mu_{j-1}}\\
				&=& d^{(\nu)} (j-1) (N-j+1) \tfrac{\mu_{j-1}}{\mu_j}.
		\end{eqnarray*}}
		Thus, we can rewrite the above equation as follows 
		{\small \begin{eqnarray*}
				& & \Big((n+i-j)(d^{(\nu)} j + c^{(\nu)}) + d^{(\nu)} (j-1) (N-j+1) + n (d^{(\nu)} (i-N-1)-c^{(\nu)}) \Big) \xi(n,i,j)\\
				& & + d^{(\nu)} (j-1) (N-j+1) \tfrac{\mu_{j-1}}{\mu_j} \xi(n,i,j-1) +  (n+i-j) (d^{(\nu)} j + c^{(\nu)}) \xi(n,i,j+1) = 0.
		\end{eqnarray*}}
		Now, if we consider $M_j=(n+i-j) (d^{(\nu)} j + c^{(\nu)})$. 
		Since $\epsilon^{(n,i)}_j \xi(n,i,j) = q_j^{(n,i)}$,
		we can rewrite the above equation as follows
		{\small \begin{eqnarray*}
				\Big((n+i-j)(d^{(\nu)} j + c^{(\nu)}) + d^{(\nu)} (j-1) (N-j+1) + n (d^{(\nu)} (i-N-1)-c^{(\nu)}) \Big) \tfrac{q_j}{\epsilon_j} \\ 
				+ d^{(\nu)} (j-1) (N-j+1) \tfrac{\mu_{j-1}}{\mu_j} \tfrac{q_{j-1}}{\epsilon_{j-1}} +  M_j \tfrac{q_{j+1}}{\epsilon_{j+1}} = 0.
		\end{eqnarray*}}
		If we multiply by $\epsilon_j$, we have that
		{\small \begin{eqnarray*}
				\Big((n+i-j)(d^{(\nu)} j + c^{(\nu)}) + d^{(\nu)} (j-1) (N-j+1) + n (d^{(\nu)} (i-N-1)-c^{(\nu)}) \Big) q_j \\ 
				+ d^{(\nu)} (j-1) (N-j+1) \tfrac{\mu_{j-1}}{\mu_j} \tfrac{\epsilon_j}{\epsilon_{j-1}} q_{j-1} +  M_j \tfrac{\epsilon_j}{\epsilon_{j+1}} q_{j+1} = 0.
		\end{eqnarray*}}
		By Lemma \ref{lema epsilon's}, we have that $M_j \tfrac{\epsilon_j}{\epsilon_{j+1}} =1 $ and so we obtain 
		{\small\begin{eqnarray*}
				&&d^{(\nu)} \Big((n+i-j)( j + \tfrac{c^{(\nu)}}{d^{(\nu)}}) +  (j-1) (N-j+1) + n  ( (i-N-1)-\tfrac{c^{(\nu)}}{d^{(\nu)}}) \Big) q_j \\
				&+& d^{(\nu)} (j-1) (N-j+1) \tfrac{\mu_{j-1}}{\mu_j}  (n+i-j+1)  \big(d^{(\nu)} (j-1) + c^{(\nu)} \big)  q_{j-1} +   q_{j+1} = 0,
		\end{eqnarray*}}
		which is equivalent to 
		{\small $$ E_j q_j +  F_j  q_{j-1} + \tfrac{1}{d^{(\nu)}}   q_{j+1} = 0,$$}
		with
		{\small \begin{align*}
				E_j &=(n+i-j) ( j + \tfrac{c^{(\nu)}}{d^{(\nu)}}) +  (j-1) (N-j+1) + n  ( (i-N-1)-\tfrac{c^{(\nu)}}{d^{(\nu)}}),\\
				F_j &=(j-1) (N-j+1) \tfrac{\mu_{j-1}}{\mu_j}  (n+i-j+1)  \big(d^{(\nu)} (j-1) + c^{(\nu)} \big),    
		\end{align*}}
		as asserted.
	\end{proof}

\end{document}